\documentclass[preprint,3p,times]{elsarticle}
\usepackage{amsmath,amssymb,amsfonts,mathrsfs}
\usepackage{amsthm}
\usepackage[hidelinks]{hyperref}
\usepackage{enumitem}
\usepackage{array}
\usepackage{booktabs}
\usepackage{xcolor}
\newcommand{\rev}[1]{#1}
\hypersetup{
  pdftitle={Representation theorems for four classes of uninorms on bounded trellises},
  pdfauthor={Zhenyu Xiu},
  pdfkeywords={uninorm, bounded trellis, pseudo-order, directed reachability, interior operator, closure operator, representation theorem}
}
\biboptions{sort&compress}

\theoremstyle{plain}
\newtheorem{theorem}{Theorem}[section]
\newtheorem{lemma}[theorem]{Lemma}
\newtheorem{proposition}[theorem]{Proposition}
\newtheorem{corollary}[theorem]{Corollary}
\theoremstyle{definition}
\newtheorem{definition}[theorem]{Definition}
\newtheorem{example}[theorem]{Example}
\newtheorem{remark}[theorem]{Remark}

\journal{Fuzzy Sets and Systems}
\begin{document}
\begin{frontmatter}

\title{
Representation theorems for four classes of uninorms on bounded
trellises
\tnoteref{funding}}
\tnotetext[funding]{This work was supported by the National Natural Science Foundation of China under Grant Nos. 12671553 and 12271036.}
\author[label2]{Zhenyu Xiu\corref{cor1}}
\cortext[cor1]{Corresponding author.}
\ead{xyz198202@163.com}

\begin{abstract}

In this paper, we establish necessary and sufficient representation
theorems for the classes
$\widetilde{\mathcal{U}}_{\mathrm{top}}$,
$\widetilde{\mathcal{U}}_{\mathrm{bot}}$,
$\widetilde{\mathcal{U}}_{\mathrm{max}}^{r}$, and
$\widetilde{\mathcal{U}}_{\mathrm{min}}^{r}$ of uninorms on bounded
trellises, thereby extending the corresponding results for bounded
lattices. For a non-extremal neutral element $e$ belonging to no
non-trivial cycle, we systematically use a common decomposition of the
elements incomparable with $e$, exploiting the fact that every neutral
element of a uninorm is middle-transitive. We show that each uninorm in
$\widetilde{\mathcal{U}}_{\mathrm{top}}$ (respectively,
$\widetilde{\mathcal{U}}_{\mathrm{bot}}$) can be represented by an
interior operator (respectively, a closure operator), a component
uninorm, and an increasing, associative, and commutative operation.
Moreover, all these components are uniquely determined. Adding two
conditions in each case yields corresponding representations of the
subclasses $\widetilde{\mathcal{U}}_{\mathrm{top}}^{\star}$ and
$\widetilde{\mathcal{U}}_{\mathrm{bot}}^{\star}$.
For $\widetilde{\mathcal{U}}_{\mathrm{max}}^{r}$ (respectively,
$\widetilde{\mathcal{U}}_{\mathrm{min}}^{r}$), we identify the pairs in
$I_e^3\times I_e^3$ whose values lie in
$[0,e[\cup I_e^1$ (respectively, $]e,1]\cup I_e^2$), as well as those
whose value equals the neutral element $e$. The pairs producing $e$
satisfy a symmetric unique-partner condition, whereas the remaining
values are described by a partial operation. Together with regional
order compatibility and explicit associativity conditions, this
decomposition yields necessary and sufficient representations of both
classes.
 Finally, we compare the four classes, clarify their relationship with the
corresponding bounded-psoset classes, and show that their representations
specialize to their bounded-lattice counterparts under transitivity;
finite examples on proper
trellises illustrate the role of the additional compatibility  conditions.

\end{abstract}

\end{frontmatter}

\section{Introduction}\label{sec:introduction}

In 1996, Yager and Rybalov~\cite{RR96} introduced uninorms on the real unit
interval as a class of aggregation functions that generalizes both triangular
norms and triangular conorms. Their structure was subsequently developed by
Fodor, Yager, and Rybalov~\cite{FRR}, and their role in fuzzy-system modeling
was studied by Yager~\cite{Y01}. These developments form part of the broader
study of fuzzy aggregation and many-valued reasoning
\cite{DD85,DD00,PH98}. On the unit interval, single-point characterizations
of representable uninorms and ordinal-sum representation results were
established in~\cite{FodorDeBaets12,Mesiarova17a,Mesiarova17b}.

In 2015, Kara\c{c}al and Mesiar~\cite{FK15} extended uninorms from the real
unit interval to arbitrary bounded lattices. In this setting, elements
incomparable with the neutral element give rise to mixed regions that do not
occur on a chain.
The study of uninorms was subsequently extended from the real unit
interval to bounded lattices~\cite{BK14,FK15}.
Subsequent work developed
systematic construction and characterization methods
combined
t-norms, t-conorms, and other component operations on principal intervals
\cite{BK14,CKM16,FK17}. Further structural results and construction methods
made the role of the incomparable region increasingly explicit
\cite{Cayli19,Cayli19alt,DHQ19,DHu20,GD20,AM21,Cayli21}, while extremal
extensions and additional class characterizations were obtained in
\cite{XL20,SL22}.

Several other approaches have also been proposed. Ordinal sums were extended
from chains to bounded lattices and later studied within a unified framework
of partially ordered monoids~\cite{Saminger06,Dvorak22}. Closure and interior
operators were used to control mixed regions in uninorm constructions
\cite{OZ20,ZW21,CayliCI23}; related interior-operator methods were also used
for ordinal sums of t-norms~\cite{Ouyang22}. T-subnorms, t-subconorms, and
additive generators provided alternative component operations
\cite{Ji21,HJ22,HW21}. More recent methods start from uninorms on subintervals
or use comparable operators and triangular conorms
\cite{CEK23,XZ23,XZ25,XS26}. The representation of idempotent and internal
uninorms has provided further structural information
\cite{OZWB22,Ouyang24}, and recent results show explicitly that the set of
elements incomparable with the neutral element is essential
\cite{GGPF26}.

The characterization of specific classes of uninorms is particularly
relevant to the present work. Zhang et al.~\cite{HP21} characterized the
bounded-lattice classes $\mathcal{U}_{\mathrm{min}}$ and
$\mathcal{U}_{\mathrm{max}}$ and introduced several related subclasses. The
classes $\mathcal{U}_{\mathrm{min}}^{1}$ and
$\mathcal{U}_{\mathrm{max}}^{0}$ were subsequently represented
in~\cite{ZO26}.   Complete representation theorems for
$\mathcal{U}_{\mathrm{top}}$ and $\mathcal{U}_{\mathrm{bot}}$, and for
$\mathcal{U}_{\mathrm{max}}^{r}$ and
$\mathcal{U}_{\mathrm{min}}^{r}$, on bounded lattices were obtained
in~\cite{ZQ26tb,ZQ26r}. On bounded pseudo-ordered sets,
Kong and Liu~\cite{KL26} introduced reachability-based extensions of several
of these classes and provided corresponding construction and existence
results.  These works form the starting point of the present study.  The
question addressed here is how to obtain complete representations,
with unique recovery of all components, when the underlying trellis
relation is not transitive.


A trellis retains binary meets and joins but replaces the partial order of a
lattice by a reflexive and antisymmetric relation that is not necessarily
transitive~\cite{EF70,HS71,HS72,KG73}. Pseudo-ordered sets have a completion
theory of their own~\cite{CH22}. In aggregation theory, triangular norms on
bounded trellises were introduced in~\cite{LZ23}, and ordinal-sum
constructions revealed further effects of non-transitivity~\cite{JG25}.
Uninorms and their construction on bounded trellises were studied in
\cite{YK24,luo25,SZ26}. Xiu and Zheng also introduced and completely
characterized nullnorms on bounded trellises~\cite{XZNull25}.
More recent
results on extremal classes of uninorms on bounded trellises make the effect
of non-transitivity explicit~\cite{WDB26}.

The present results are not obtained by simply replacing the partial order in
the bounded-lattice representations with the pseudo-order of a trellis. In a
bounded lattice, successive comparisons can be combined by transitivity,
whereas on a proper trellis a directed path
$x=x_0\unlhd x_1\unlhd\cdots\unlhd x_n=y$ need not imply the direct
comparison $x\unlhd y$. Consequently, the region $I_e$ of elements incomparable with the
neutral element cannot be treated uniformly; instead, it must be
decomposed into the three subsets $I_e^1$, $I_e^2$, and $I_e^3$
according to the existence and direction of finite chains of
comparisons connecting these elements with $e$.

%

Moreover, monotonicity of the component
operations on their individual domains does not guarantee monotonicity of the
assembled piecewise operation: comparisons between pairs belonging to
different formula regions require additional regional order-compatibility
conditions, together with the appropriate join or meet compatibility in the
top and bottom cases.
Although bounded pseudo-ordered sets form a more general ambient setting,
some results in~\cite{KL26} were formulated primarily as
construction and existence results.  For $\mathcal{U}_{\mathrm{max}r}^{\mathrm{ps}}$
  and   $\mathcal{U}_{\mathrm{min}r}^{\mathrm{ps}}$,
 their specialization to bounded
trellises therefore does not by itself provide the complete
representation theorems  obtained here.



Under the standing assumption $e\in K\setminus\{0,1\}$,  our purpose is  to  establish necessary and sufficient representation
theorems for the classes
$\widetilde{\mathcal{U}}_{\mathrm{top}}$,
$\widetilde{\mathcal{U}}_{\mathrm{bot}}$,
$\widetilde{\mathcal{U}}_{\mathrm{max}}^{r}$, and
$\widetilde{\mathcal{U}}_{\mathrm{min}}^{r}$ of uninorms on bounded
trellises, thereby extending the corresponding results for bounded
lattices.  That is, we will determine exactly which components are associated with each of the four
classes,   identify the compatibility conditions under which these
components assemble into a uninorm, and   prove that the components are
uniquely recoverable from the resulting operation. More specifically,  for
$\widetilde{\mathcal{U}}_{\mathrm{top}}$ and
$\widetilde{\mathcal{U}}_{\mathrm{bot}}$, we obtain representations in terms
of an interior or closure operator and suitable component operations.
For $\widetilde{\mathcal{U}}_{\mathrm{max}}^{r}$ (respectively,
$\widetilde{\mathcal{U}}_{\mathrm{min}}^{r}$), we identify the pairs in
$I_e^3\times I_e^3$ whose values lie in $A_e$ (respectively, $B_e$) or
equal $e$; the remaining values are represented by a partial operation
defined on the corresponding residual domain.
For
$\widetilde{\mathcal{U}}_{\mathrm{max}}^{r}$ and
$\widetilde{\mathcal{U}}_{\mathrm{min}}^{r}$, we give exact piecewise
representations of the remaining behavior on
$I_e^3\times I_e^3$.
These four principal classes form two dual pairs and exhibit two
different representation patterns, while all four are governed by the
same decomposition associated with the neutral element.
 When the pseudo-order is transitive, these representations
reduce to their bounded-lattice counterparts; the proper-trellis examples illustrate the role of the additional
conditions and show, in particular, that regional order compatibility
cannot in general be omitted.



The remainder of this paper is organized as follows.
Section~\ref{sec:preliminaries} recalls the basic concepts and introduces
the regions determined by the neutral element.
Section~\ref{sec:part-one} establishes construction and representation
theorems for
$\widetilde{\mathcal{U}}_{\mathrm{top}}$ and
$\widetilde{\mathcal{U}}_{\mathrm{bot}}$, characterizes the subclasses
$\widetilde{\mathcal{U}}_{\mathrm{top}}^{\star}$ and
$\widetilde{\mathcal{U}}_{\mathrm{bot}}^{\star}$, and gives finite examples.
Section~\ref{sec:r-structural-properties} classifies the distinguished
pairs in $I_e^3\times I_e^3$, derives piecewise representations with
unique recovery of their components, and provides finite examples.
Section~\ref{sec:comparison} compares the four classes, relates them to
the corresponding classes on bounded psosets,  and establishes their
bounded-lattice specialization under transitivity.
%
Finally, Section~\ref{sec:conclusion} summarizes the main results and
presents directions for future research.

\section{Preliminaries}
\label{sec:preliminaries}

In this section, we recall the concepts of pseudo-orders, bounded trellises, lattices,
 interior and closure operators, t-norms, t-conorms, and
uninorms. We also introduce the notation for the common regions used throughout the
paper.

\begin{definition}[{\cite{HS71,HS72}}]\label{def:psoset}
A binary relation $\unlhd$ on a set $X$ is called a \emph{pseudo-order} if it
is reflexive and antisymmetric. The pair $\mathbb{S}=(X,\unlhd)$ is called a
\emph{pseudo-ordered set}, or a \emph{psoset}.
\end{definition}

We write $a\lhd b$ if $a\unlhd b$ and $a\ne b$, and write
$a\ntrianglelefteq b$ if $a\unlhd b$ does not hold. Elements $a,b$ are
\emph{incomparable}, denoted $a\parallel b$, if neither $a\unlhd b$ nor
$b\unlhd a$ holds, and $I_a=\{x\in X:x\parallel a\}$. A subset $L\subseteq
X$ is a \emph{lower set} if $y\in L$ and $x\unlhd y$ imply $x\in L$. A
subset $H\subseteq X$ is an \emph{upper set} if $x\in H$ and $x\unlhd y$
imply $y\in H$. A map between sub-psosets is \emph{order-preserving} if
it preserves every direct comparison.

We use the standard notions of minimal and maximal elements, smallest and
greatest elements, lower and upper bounds, infimum, and supremum as for
partially ordered sets~\cite{Gierz2003}. Antisymmetry guarantees uniqueness
whenever an infimum or a supremum exists. For a two-element set $\{a,b\}$,
these are denoted by $a\wedge b$ and $a\vee b$, respectively.

For $x,y\in X$, write $x\lesssim y$ if there is a finite sequence
$x=x_0\unlhd x_1\unlhd\cdots\unlhd x_n=y$. A nonempty subset $C$ is a
\emph{cycle} if every two elements are connected in both directions by such
sequences whose terms all lie in $C$. A
singleton cycle is trivial; antisymmetry implies that every non-trivial
cycle has at least three elements.

Following~\cite{LZ23}, we represent finite psosets by Hasse-type
diagrams. In all finite examples considered below, the strict part of
$\unlhd$ is acyclic; hence $\lesssim$ is a partial order. The cover
relations of $\lesssim$ are represented by upward solid edges; thus, an
upward solid path from $x$ to $y$ represents $x\lesssim y$. If
$x\lesssim y$ but $x\parallel y$ under $\unlhd$, then $x$ and $y$ are
joined by a dashed curve, indicating that the solid path does not correspond
to a direct comparison under $\unlhd$. Reflexive loops and direct
$\unlhd$-comparisons that are not covers of $\lesssim$ are not drawn
separately. The neutral element $e$ is shown as a shaded vertex.

\begin{definition}[{\cite{YK24}}]\label{def:bounded-psoset}
A psoset is called \emph{bounded} if it has a smallest element $0$ and a greatest
element $1$, so $0\unlhd x\unlhd1$ for every $x\in X$.
\end{definition}

\begin{definition}[{\cite{HS71,HS72,KG73}}]\label{def:trellis}
A psoset is called a $\wedge$-semi-trellis (resp.\
$\vee$-semi-trellis) if every pair has a meet (resp.\ a join). A psoset that
is both a $\wedge$-semi-trellis and a $\vee$-semi-trellis is called a
\emph{trellis} and is denoted by
$\mathbb{P}=(X,\unlhd,\wedge,\vee)$. In a trellis,
$a\unlhd b$ is equivalent to $a\wedge b=a$, or equivalently to
$a\vee b=b$. A bounded trellis is written
$\mathbb{P}=(X,\unlhd,\wedge,\vee,0,1)$ and is assumed here to have at
least three elements. If $\unlhd$ is transitive, then it is
a partial order and the trellis $\mathbb{P}$ is called a \emph{lattice}. A
bounded trellis whose pseudo-order is transitive is called a
\emph{bounded lattice}. A bounded trellis is called \emph{proper} if
$\unlhd$ is not transitive.
\end{definition}

Throughout this paper, $\mathbb{P}$ denotes the  bounded trellis.
To distinguish the trellis classes introduced below from their counterparts
on bounded lattices, we use
$\widetilde{\mathcal{U}}_{\mathrm{top}}$,
$\widetilde{\mathcal{U}}_{\mathrm{bot}}$,
$\widetilde{\mathcal{U}}_{\mathrm{max}}^{r}$, and
$\widetilde{\mathcal{U}}_{\mathrm{min}}^{r}$
for the former and reserve the corresponding untilded symbols for the latter.

\begin{definition}[{\cite{LZ23}}]\label{def:transitive-elements}
An element $a$ of a psoset is called \emph{right-transitive},
\emph{left-transitive}, or \emph{middle-transitive} if, respectively,
\[
 a\unlhd x\unlhd y\Rightarrow a\unlhd y,\qquad
 x\unlhd y\unlhd a\Rightarrow x\unlhd a,
\]
or
\[
 x\unlhd a\unlhd y\Rightarrow x\unlhd y.
\]
The corresponding sets of elements are denoted by $X^{\mathrm{rtr}}$,
$X^{\mathrm{ltr}}$, and $X^{\mathrm{mtr}}$. An element
$a$ is called transitive if it is right-, left-, and middle-transitive.
\end{definition}

\begin{definition}[{\cite{YK24}}]\label{def:interval}
Let $a,b\in X$ with $a\unlhd b$. The interval $[a,b]$ is defined by
\[
[a,b]=\{x\in X:a\unlhd x\unlhd b\}.
\]
The intervals $[a,b[$, $]a,b]$, and $]a,b[$ are defined analogously by
making the corresponding endpoint inequality strict.
For $a\in X$, define
\[
 I_a^1=\{x\in I_a:x\lesssim a\},\qquad
 I_a^2=\{x\in I_a:a\lesssim x\},
\]
and
\[
 I_a^3=\{x\in I_a:x\not\lesssim a\text{ and }a\not\lesssim x\}.
\]
Moreover, put $K=\{a\in X^{\mathrm{mtr}}:a\text{ belongs to no non-trivial cycle}\}$.

\end{definition}

\begin{lemma}[{\cite{XZNull25}}]\label{lem:reachability-partition}
If $a\in K$, then $I_a^1\cap I_a^2=\varnothing$.
\end{lemma}

\begin{definition}[{\cite{LZ23}}]\label{def:operation-properties}
On $X\times X$, define the product pseudo-order by
\[
 (x_1,x_2)\preceq(y_1,y_2)
 \quad\Leftrightarrow\quad
 x_1\unlhd y_1\text{ and }x_2\unlhd y_2.
\]
For $P\subseteq X\times X$, a map $F:P\to X$ is called \emph{increasing} (or
jointly increasing with respect to $\preceq$) if
$\boldsymbol{x}\preceq\boldsymbol{y}$ for
$\boldsymbol{x},\boldsymbol{y}\in P$ implies
$F(\boldsymbol{x})\unlhd F(\boldsymbol{y})$. Throughout this paper,
``increasing'' for a total or partial binary map always has this joint
meaning; it does not mean merely that its one-variable sections are
order-preserving.
\end{definition}

Let $(Y,\unlhd,0_Y,1_Y)$ be a bounded psoset. A binary operation on $Y$ is
called a \emph{t-norm} (resp.\ a \emph{t-conorm}) if it is commutative,
associative, and jointly increasing with respect to the product pseudo-order
on $Y\times Y$, and has neutral element $1_Y$ (resp.\ $0_Y$)~\cite{LZ23}.

\begin{definition}[{\cite{Ouyang22}}]\label{def:interior-operator}
Let $(Y,\unlhd)$ be a psoset. A self-map $\iota:Y\to Y$ is called an
\emph{interior operator} if it is order-preserving,
$\iota(y)\unlhd y$, and $\iota^2(y)=\iota(y)$ for every $y\in Y$.
\end{definition}

\begin{definition}[{\cite{Erne09}}]\label{def:closure-operator}
Let $(Y,\unlhd)$ be a psoset. A self-map $\kappa:Y\to Y$ is called a
\emph{closure operator} if it is order-preserving,
$y\unlhd\kappa(y)$, and $\kappa^2(y)=\kappa(y)$ for every $y\in Y$.
\end{definition}

\begin{definition}[{\cite{YK24}}]\label{def:uninorm}
Let $(Y,\unlhd)$ be a psoset and let $e\in Y$. A binary operation
$U:Y\times Y\to Y$ is called a \emph{uninorm with neutral element $e$} if it is
commutative, associative, and jointly increasing with respect to the product
pseudo-order on $Y\times Y$, and satisfies $U(y,e)=y$ for every $y\in Y$.
\end{definition}

\begin{proposition}\label{prop:neutral-middle-transitive}
The neutral element of every uninorm on a psoset is middle-transitive.
\end{proposition}

\begin{proof}
Let $U$ be a uninorm with neutral element $e$, and suppose that
$x\unlhd e\unlhd y$. Then $(x,e)\preceq(e,y)$. Joint monotonicity,
commutativity, and neutrality give
$
x=U(x,e)\unlhd U(e,y)=U(y,e)=y.
$
Thus $e$ is middle-transitive.
\end{proof}

Consequently, middle-transitivity of the neutral element follows from the
uninorm axioms under the joint notion of monotonicity used here. By contrast,
being distinct from the bounds and belonging to no non-trivial cycle are
additional restrictions. Thus the standing assumption
$e\in K\setminus\{0,1\}$ specifies precisely the non-extremal neutral elements that belong to no non-trivial cycle considered below.


Unless otherwise stated, fix $e\in K\setminus\{0,1\}$. We use the
following basic regions:
\[
A_e=[0,e[\cup I_e^1,\qquad
B_e=]e,1]\cup I_e^2,
\]
whereas $I_e^3$ is the remaining incomparable region. Their enlargements
are defined by
\[
\widehat A_e=A_e\cup I_e^3,\qquad
\widehat B_e=B_e\cup I_e^3.
\]
\rev{Their equivalent descriptions in terms of $\lesssim$ are proved in
Lemma~\ref{lem:r-reachability}. Since $e$ belongs to no non-trivial cycle,
these regions form the pairwise disjoint decomposition}
$
\rev{X=A_e\mathbin{\cup}\{e\}\mathbin{\cup}
B_e\mathbin{\cup}I_e^3.}
$
Each of these subsets is endowed with the pseudo-order induced from
$X$. In the notation of~\cite{KL26}, the first two regions are denoted by
$
A_e=[0,e[_{\lesssim}
\ \text{and}\
B_e=]e,1]_{\lesssim}.
$

\begin{lemma}\label{lem:r-reachability}
Let $H:X\times X\to X$ be increasing. If $x\lesssim y$ and $z\lesssim t$, then
$H(x,z)\lesssim H(y,t)$. Moreover,
$
 A_e=\{x\in X\setminus\{e\}\mid x\lesssim e\}$,
   and
 $B_e=\{x\in X\setminus\{e\}\mid e\lesssim x\}.
$
\end{lemma}

\begin{proof}
Choose paths
\[
x=x_0\unlhd x_1\unlhd\cdots\unlhd x_m=y,
\qquad
z=z_0\unlhd z_1\unlhd\cdots\unlhd z_n=t.
\]
Since $H$ is jointly increasing, the comparisons
$(x_i,z)\preceq(x_{i+1},z)$ and
$(y,z_j)\preceq(y,z_{j+1})$ give the path
\[
H(x,z)=H(x_0,z)\unlhd\cdots\unlhd H(x_m,z)
=H(y,z_0)\unlhd\cdots\unlhd H(y,z_n)=H(y,t).
\]
Hence $H(x,z)\lesssim H(y,t)$.

If $x\in A_e$, then either $x\in[0,e[$ or $x\in I_e^1$, and in both
cases $x\ne e$ and $x\lesssim e$. Conversely, let $x\ne e$ and
$x\lesssim e$. If $x\unlhd e$, then $x\in[0,e[$. Otherwise,
$e\unlhd x$ is impossible, because this edge together with a path from
$x$ to $e$ would place $e$ in a non-trivial cycle. Hence $x\parallel e$,
and $x\lesssim e$ gives $x\in I_e^1$. This proves the identity for $A_e$.

The argument for $B_e$ is analogous. If $x\in B_e$, then
$e\lesssim x$. Conversely, if $x\ne e$ and $e\lesssim x$, then either
$e\unlhd x$, so $x\in]e,1]$, or $e\ntrianglelefteq x$. In the latter case,
$x\unlhd e$ would form a non-trivial cycle with a path from $e$ to $x$;
therefore $x\parallel e$ and $x\in I_e^2$. This proves the identity for
$B_e$.
\end{proof}

%

\section{Part I: \texorpdfstring{Representations}{Representations} of the classes
\texorpdfstring{$\widetilde{\mathcal{U}}_{\mathrm{top}}$ and $\widetilde{\mathcal{U}}_{\mathrm{bot}}$}
{Utop and Ubot}}
\label{sec:part-one}

In this part, we introduce the trellis classes
$\widetilde{\mathcal{U}}_{\mathrm{top}}$ and
$\widetilde{\mathcal{U}}_{\mathrm{bot}}$ by means of
the common regions. We then establish construction and representation theorems for both
classes and show that the component maps are uniquely determined. Finally, we
consider the subclasses
$\widetilde{\mathcal{U}}_{\mathrm{top}}^{\star}$ and
$\widetilde{\mathcal{U}}_{\mathrm{bot}}^{\star}$.

\subsection{Definitions in the common notation}

The following definitions adapt the bounded-lattice top and bottom classes
studied in~\cite{ZQ26tb} to the regions of a bounded trellis. The starred
refinements are consistent with the corresponding psoset subclasses
in~\cite{KL26}.

\begin{definition}\label{def:top-bottom-classes}
Let $e\in K\setminus\{0,1\}$.
We define $\widetilde{\mathcal{U}}_{\mathrm{top}}$ as the class of all
uninorms $U$ on $\mathbb{P}$ with neutral element $e$ satisfying
\begin{itemize}
  \item[$(\circledast)$] For every
  $x\in \widehat B_e$, there exists
  $q_x\in \widehat B_e$ such that
  \[
  U(x,y)=q_x
  \qquad\text{for all }y\in A_e.
  \]
\end{itemize}

Similarly, we define $\widetilde{\mathcal{U}}_{\mathrm{bot}}$ as the class of
all uninorms $U$ on $\mathbb{P}$ with neutral element $e$ satisfying
\begin{itemize}
  \item[$(\circledcirc)$] For every
  $x\in\widehat A_e$, there exists
  $p_x\in\widehat A_e$ such that
  \[
  U(x,y)=p_x
  \qquad\text{for all }y\in B_e.
  \]
\end{itemize}

In addition, $\widetilde{\mathcal{U}}_{\mathrm{top}}^{\star}$ is the subclass
of $\widetilde{\mathcal{U}}_{\mathrm{top}}$ consisting of all uninorms $U$
satisfying
\begin{enumerate}[label=\textup{(A\arabic*)}, itemsep=1pt,
                  parsep=0pt, topsep=8pt]
  \item For every $x\in I_e^2\cup I_e^3$,
  \[
  U(x,y)=x
  \qquad\text{for all }y\in A_e.
  \]
  \item For every $x\in]e,1]$, there exists $m_x\in]e,1]$ such that
  \[
  U(x,y)=m_x
  \qquad\text{for all }y\in A_e.
  \]
  \item For all $x\in I_e^2\cup I_e^3$ and $y\in]e,1]$,
  \[
  x\vee y=x\vee m_y.
  \]
\end{enumerate}

Dually, $\widetilde{\mathcal{U}}_{\mathrm{bot}}^{\star}$ is the subclass of
$\widetilde{\mathcal{U}}_{\mathrm{bot}}$ consisting of all uninorms $U$
satisfying
\begin{enumerate}[label=\textup{(B\arabic*)}, itemsep=1pt,
                  parsep=0pt, topsep=8pt]
  \item For every $x\in I_e^1\cup I_e^3$,
  \[
  U(x,y)=x
  \qquad\text{for all }y\in B_e.
  \]
  \item For every $x\in[0,e[$, there exists $n_x\in[0,e[$ such that
  \[
  U(x,y)=n_x
  \qquad\text{for all }y\in B_e.
  \]
  \item For all $x\in I_e^1\cup I_e^3$ and $y\in[0,e[$,
  \[
  x\wedge y=x\wedge n_y.
  \]
\end{enumerate}
\end{definition}

If $\unlhd$ is transitive, then
$I_e^1=I_e^2=\varnothing$ and $I_e^3=I_e$. Hence the definitions of
$\widetilde{\mathcal{U}}_{\mathrm{top}}^{\star}$ and
$\widetilde{\mathcal{U}}_{\mathrm{bot}}^{\star}$ above coincide with the
corresponding bounded-lattice classes
$\mathcal{U}_{\mathrm{top}}^{\star}$ and
$\mathcal{U}_{\mathrm{bot}}^{\star}$, respectively.

The following immediate consequences of monotonicity will be used below.
If $U\in\widetilde{\mathcal{U}}_{\mathrm{top}}$, then
$q_x\unlhd x$ for every $x\in\widehat B_e$; if
$U\in\widetilde{\mathcal{U}}_{\mathrm{bot}}$, then
$x\unlhd p_x$ for every $x\in\widehat A_e$. Thus,
\begin{equation}\label{eq:top-bottom-boundary-values}
q_x\unlhd x\quad(x\in\widehat B_e),
\qquad
x\unlhd p_x\quad(x\in\widehat A_e).
\end{equation}
Similarly, for
$\widetilde{\mathcal{U}}_{\mathrm{top}}^{\star}$ and
$\widetilde{\mathcal{U}}_{\mathrm{bot}}^{\star}$, respectively,
monotonicity yields
\[
m_x\unlhd x\quad(x\in]e,1]),
\qquad
x\unlhd n_x\quad(x\in[0,e[).
\]

The notation $A_e$, $B_e$, $I_e^3$, $\widehat A_e$, and $\widehat B_e$
is the common notation fixed in Section~\ref{sec:preliminaries}.

\begin{lemma}\label{lem:top-block-closure}
Let $e\in K\setminus\{0,1\}$.
Then the sets
$
A_e,\  \{e\},\  \widehat B_e
$
form a partition of $X$.
Moreover,
$\widehat B_e$ is an upper set: if $d\in \widehat B_e$ and $d\unlhd z$,
then $z\in \widehat B_e$. Finally, if $V:X\times X\to X$ is an increasing operation
with neutral element $e$, then $V(A_e\times A_e)\subseteq A_e$.
\end{lemma}

\begin{proof}
By Lemma \ref{lem:reachability-partition}, $I_e^1\cap I_e^2=\varnothing$.
Hence, by the definitions of
$A_e$ and $\widehat B_e$, the sets
$
A_e,\  \{e\},\  \widehat B_e
$
form a partition of $X$.

Let $d\in \widehat B_e$ and $d\unlhd z$. If $z\in A_e\cup\{e\}$, then
$d\lesssim e$. This contradicts the definition of $I_e^3$ when
$d\in I_e^3$; when $d\in B_e$, it combines with
$e\lesssim d$ to produce a non-trivial cycle containing $e$.
Therefore $z\in \widehat B_e$.

Now let $x,y\in A_e$. There is a directed path from $y$ to $e$.
No element of this path can lie in $\widehat B_e$, since such an element would
satisfy $d\lesssim e$, contrary to the preceding argument. Thus all
its intermediate elements lie in $A_e\cup\{e\}$. Monotonicity of $V$
along this path gives
$V(x,y)\lesssim V(x,e)=x\lesssim e$. If $V(x,y)=e$, then
$e\lesssim x\lesssim e$ yields a non-trivial cycle containing $e$;
if $V(x,y)\in \widehat B_e$, the preceding argument again gives a contradiction.
The disjoint decomposition therefore implies $V(x,y)\in A_e$.
\end{proof}

\subsection{Construction and representation of
\texorpdfstring{$\widetilde{\mathcal{U}}_{\mathrm{top}}$}{Utop}}

We first provide a construction of uninorms in
$\widetilde{\mathcal{U}}_{\mathrm{top}}$ from an interior operator, a
component uninorm, and an increasing, associative, and commutative operation.
The converse representation is established afterward.

\begin{theorem}[Construction of $\widetilde{\mathcal{U}}_{\mathrm{top}}$]\label{thm:top-construction}
Let $\mathbb{P}=(X,\unlhd,\wedge,\vee,0,1)$ be a bounded trellis, and let
$e\in K\setminus\{0,1\}$. Consider an interior operator
$\iota:\widehat B_e\to\widehat B_e$, a uninorm
$U_L:(A_e\cup\{e\})\times(A_e\cup\{e\})\to A_e\cup\{e\}$ with neutral element $e$, and an
increasing, associative, and commutative operation
$Q:\widehat B_e\times\widehat B_e\to\widehat B_e$. Assume that the following conditions
hold:
\begin{enumerate}[itemsep=1pt, parsep=0pt, topsep=8pt]
  \item[$\mathrm{(i)}$]  $x,y\in A_e$, $z\in \widehat B_e$ with $x\unlhd z \Rightarrow$ \(U_L(x,y)\lhd \iota(z)\);
  \item[$\mathrm{(ii)}$]  $x\in A_e$ and $y\in ]e,1] \Rightarrow x\lhd \iota(y)$;
  \item[$\mathrm{(iii)}$]  $x,r\in A_e$ and $y,s\in \widehat B_e$ with $x\unlhd y $ and $ r\unlhd s \Rightarrow$ \(U_L(x,r)\unlhd Q(y,s)\);
  \item[$\mathrm{(iv)}$]  $x\in A_e$, $y\in \widehat B_e$ and $s\in ]e,1]$ with $x\unlhd y \Rightarrow$ \(x\unlhd Q(y,s)\);
  \item[$\mathrm{(v)}$]  $x,y\in \widehat B_e$ with $x\unlhd y \Rightarrow$ \(\iota(x)\unlhd y\);
  \item[$\mathrm{(vi)}$]  $x,y,s\in \widehat B_e$ with $x\unlhd y \Rightarrow$ \(\iota(x)\unlhd Q(y,s)\);
  \item[$\mathrm{(vii)}$]  $x,y\in ]e,1] \Rightarrow$ \(e\unlhd Q(x,y)\);
  \item[$\mathrm{(viii)}$]  $x,y\in \widehat B_e$ and $s\in ]e,1]$ with $x\unlhd y \Rightarrow$ \(x\unlhd Q(y,s)\);
  \item[$\mathrm{(ix)}$]  $x,y\in \widehat B_e \Rightarrow$ \(\iota(Q(x,y))=Q(\iota(x),y)\);
  \item[$\mathrm{(x)}$]  $x\in [0,e[$, $y\in A_e$ and $z\in \widehat B_e$ with $y\unlhd z \Rightarrow$ $U_L(x,y)\lhd z$.
\end{enumerate}
The absence of the hypothesis $x\unlhd y$ in \textup{(ii)} is
essential: $(x,e)\preceq(x,y)$ for every $x\in A_e$ and
$y\in]e,1]$.
Define $U:X\times X\to X$ by
\begin{equation}\label{eq:top-construction}
\begin{aligned}
U(x,y) =
\begin{cases}
U_L(x,y), & \text{if } (x,y) \in (A_e\cup\{e\})\times(A_e\cup\{e\}), \\
x, & \text{if } (x,y) \in (\widehat B_e)\times \{e\}, \\
y, & \text{if } (x,y) \in \{e\}\times(\widehat B_e), \\
\iota(x), & \text{if } (x,y) \in (\widehat B_e)\times (A_e), \\
\iota(y), & \text{if } (x,y) \in (A_e)\times(\widehat B_e), \\
Q(x,y), & \text{if } (x,y) \in \widehat B_e\times\widehat B_e. \\
\end{cases}
\end{aligned}
\end{equation}
Then $U\in \widetilde{\mathcal{U}}_{\mathrm{top}}$.
\end{theorem}
\begin{proof}
By the definition of $U$, it is straightforward to verify that $U$ is
commutative and that $e$ is its neutral element. Thus, it remains to prove
that $U$ is increasing and associative.

I. Monotonicity:
To prove that \( U(x,r) \unlhd U(y, s) \), consider \( x,y, r, s \in X \) satisfying \( x\unlhd y \) and \( r \unlhd s \).
The argument splits into the following cases:

1. $x,r\in A_e$. Then $y,s\in X$.

\ \ 1.1. $y,s\in A_e$. Then
\(U(x,r ) = U_L(x,r) \unlhd U_L(y,s) = U (y,s)\).

\ \ 1.2. $y\in A_e$ and $s=e$. Then
\(U(x,r ) = U_L(x,r) \unlhd U_L(y,s) = U (y,s)\).
The case $s\in A_e$ and $y=e$ is similar to 1.2.

\ \ 1.3. $y\in A_e$ and $s\in \widehat B_e$. Then, by condition $\mathrm{(i)}$,
\(U(x,r ) = U_L(x,r) \unlhd \iota(s) = U (y,s)\).
The case $s\in A_e$ and
$y\in \widehat B_e$ is similar to 1.3.

\ \ 1.4. $y=s=e$. Then
\(U(x,r ) = U_L(x,r) \unlhd U_L(y,s) = U (y,s)\).

\ \ 1.5. $y=e$ and $s\in \widehat B_e$. Here $x\unlhd e$ and
$x\in A_e$ imply $x\in[0,e[$, while $r\unlhd s$ supplies the other
comparison required in condition~\textup{(x)}. Hence
\(U(x,r ) = U_L(x,r) \unlhd s = U (y,s)\).
The case $s=e$ and $y\in \widehat B_e$ is similar to 1.5.

\ \ 1.6. $y,s\in \widehat B_e$. Then, by condition $\mathrm{(iii)}$,
\(U(x,r ) = U_L(x,r) \unlhd Q(y,s) = U (y,s)\).

2. $x\in A_e$ and $r=e$. Then $y\in X$ and $s\in[e,1]$.

\ \ 2.1. $y\in A_e\cup\{e\}$ and $s=e$. Then
\(U(x,r ) = x \unlhd y = U (y,s)\).

\ \ 2.2. $y\in A_e$ and $s\in ]e,1]$. By condition~\textup{(ii)},
\(U(x,r ) =x \lhd \iota(s) =U(y,s)\).

\ \ 2.3. $y=e$ and $s\in ]e,1]$. Since $e$ is middle-transitive,
$x\unlhd e\unlhd s$ implies $x\unlhd s$, and hence
\(U(x,r ) = x \unlhd s = U(y,s)\).

\ \ 2.4. $y\in \widehat B_e$ and $s=e$. Then
\(U(x,r ) = x \unlhd y = U(y,s)\).

\ \ 2.5. $y\in \widehat B_e$ and $s\in ]e,1]$. Then, by condition~\textup{(iv)},
\(U(x,r ) = x \unlhd Q(y,s) = U(y,s)\).
The case $r\in A_e$ and $x=e$ is similar to 2.

3. $x\in A_e$ and $r\in \widehat B_e$. Then $y\in X$ and, by Lemma~\ref{lem:top-block-closure}, $s\in \widehat B_e$.

\ \ 3.1. $y\in A_e$ and $s\in \widehat B_e$. Then
\(U(x,r ) = \iota(r) \unlhd \iota(s) = U (y,s)\).

\ \ 3.2. $y= e$ and $s\in \widehat B_e$. Then, by condition $\mathrm{(v)}$,
\(U(x,r ) = \iota(r) \unlhd s = U (y,s)\).

\ \ 3.3. $y,s\in \widehat B_e$. Then, by condition~\textup{(vi)} and the
commutativity of $Q$,
\(U(x,r ) = \iota(r) \unlhd Q (y,s) = U (y,s)\).
The case $r\in A_e$ and $x\in \widehat B_e$ is similar to 3.

4. $x=r=e$. Then $y,s\in[e,1]$.

\ \ 4.1. $y=s=e$. Then
\(U(x,r ) = e = U (y,s)\).

\ \ 4.2. $y=e$ and $s\in ]e,1]$. Then
\(U(x,r ) = e \unlhd s = U (y,s)\).
The case $s=e$ and $y\in ]e,1]$ is similar to 4.2.

\ \ 4.3. $y,s\in ]e,1]$. Then, by condition $\mathrm{(vii)}$,
\(U(x,r ) = e \unlhd Q(y,s) = U (y,s)\).

5. $x=e$ and $r\in \widehat B_e$. Then $y\in  [e,1]$ and, by Lemma~\ref{lem:top-block-closure}, $s\in \widehat B_e$.\nopagebreak[4]

\ \ 5.1. $y=e$ and $s\in \widehat B_e$. Then
\(U(x,r ) = r \unlhd s = U (y,s)\).

\ \ 5.2. $y\in ]e,1]$ and $s\in \widehat B_e$. Then, by
condition~\textup{(viii)} and the commutativity of $Q$,
\(U(x,r ) = r \unlhd Q(y,s) = U (y,s)\).
The case $r=e$ and $x\in \widehat B_e$ is similar to 5.

6. $x,r\in \widehat B_e$. Then $y,s\in \widehat B_e$ by Lemma~\ref{lem:top-block-closure}, and hence
\(U(x,r ) = Q(x,r ) \unlhd Q(y,s) = U (y,s)\).

Therefore, $U(x, r) \unlhd U(y, s)$ holds for all $x, y, r, s \in X$ with $x \unlhd y$ and $r \unlhd s$. Thus, $U$ is increasing.

II. Associativity:
Let \( x, y, z\in X\).
We need to prove that \( U(x,U(y,z))=U(U(x,y), z) \).

We first record a closure property needed below:
\begin{equation}\label{eq:top-lower-block-closure}
U_L(A_e\times A_e)\subseteq A_e.
\end{equation}
Indeed, let $x,y\in A_e$. As in the proof of
Lemma~\ref{lem:top-block-closure}, there is a directed path from $y$ to $e$
whose vertices lie in $A_e\cup\{e\}$. The monotonicity of $U_L$ along this path yields \(U_L(x,y)\lesssim U_L(x,e)=x\lesssim e\).
Since the codomain of $U_L$ is $A_e\cup\{e\}$, the only possible value
outside $A_e$ is $e$. If $U_L(x,y)=e$, then
$e\lesssim x\lesssim e$, which produces a non-trivial cycle containing
$e$, contrary to $e\in K$. This proves~\eqref{eq:top-lower-block-closure}.

If $e\in\{x,y,z\}$, associativity follows from neutrality. If all
three elements lie in $A_e$, it follows from the associativity of
$U_L$ and~\eqref{eq:top-lower-block-closure}; if all three lie in $\widehat B_e$,
it follows from the associativity of $Q$.
The remaining proof can be divided into the following two cases:

1. Let $x\in A_e$.

\ \ \ 1.1. $y \in A_e$ and $ z \in \widehat B_e$. Since $\iota$ is an interior operator on $\widehat B_e$, we have
\begin{align*}
U(x,U(y,z))
 &=U(x,\iota(z))=\iota(\iota(z))\\
 &=\iota(z)=U(U_L(x,y),z)\\
 &=U(U(x,y),z).
\end{align*}

\ \ \ 1.2. $y\in\widehat B_e$.\nopagebreak[4]

\ \ \ \ \ \ 1.2.1. $ z \in A_e$. Then, because $\iota$ is an interior operator on $\widehat B_e$, we have
\begin{align*}
U(x,U(y,z))
 &=U(x,\iota(y))=\iota(\iota(y))\\
 &=\iota(y)=\iota(\iota(y))\\
 &=U(\iota(y),z)=U(U(x,y),z).
\end{align*}

\ \ \ \ \ \ 1.2.2. $ z \in \widehat B_e$. Then, by condition~(ix) and the commutativity of $Q$, we have
\begin{align*}
U(x,U(y,z))
 &=U(x,Q(y,z))=\iota(Q(y,z))\\
 &=Q(\iota(y),z)=U(\iota(y),z)\\
 &=U(U(x,y),z).
\end{align*}

2. Let $x\in\widehat B_e$.

\ \ \ 2.1. $y\in A_e$.

\ \ \ \ \ \ 2.1.1. $ z \in A_e$. Since $\iota$ is an interior operator on $\widehat B_e$, we have
\begin{align*}
U(x,U(y,z))
 &=U(x,U_L(y,z))=\iota(x)\\
 &=\iota(\iota(x))=U(\iota(x),z)\\
 &=U(U(x,y),z).
\end{align*}

\ \ \ \ \ \ 2.1.2. $ z \in \widehat B_e$. Then, by
condition~\textup{(ix)} and the commutativity of $Q$, we have
\begin{align*}
U(x,U(y,z))
 &=U(x,\iota(z))=Q(x,\iota(z))\\
 &=Q(\iota(z),x)=\iota(Q(z,x))\\
 &=\iota(Q(x,z))=Q(\iota(x),z)\\
 &=U(\iota(x),z)=U(U(x,y),z).
\end{align*}
\ \ \ 2.2. $y\in\widehat B_e$.

\ \ \ \ \ \ 2.2.1. $ z \in A_e$. Then, by condition~(ix) and the commutativity of $Q$, we have
\begin{align*}
U(x,U(y,z))
 &=U(x,\iota(y))=Q(x,\iota(y))\\
 &=\iota(Q(x,y))=U(Q(x,y),z)\\
 &=U(U(x,y),z).
\end{align*}
\ \ \ \ \ \ 2.2.2. $ z \in \widehat B_e$. Then we have
\begin{align*}
U(x,U(y,z))
 &=U(x,Q(y,z))=Q(x,Q(y,z))\\
 &=Q(Q(x,y),z)=U(Q(x,y),z)\\
 &=U(U(x,y),z).
\end{align*}
Hence $U$ is associative.

Finally, if $d\in \widehat B_e$ and $a\in A_e$, then
$U(d,a)=\iota(d)\in \widehat B_e$. Hence condition $(\circledast)$ in
Definition~\ref{def:top-bottom-classes} holds (with $q_d=\iota(d)$), and
$U\in\widetilde{\mathcal{U}}_{\mathrm{top}}$.
\end{proof}

\begin{theorem}[Representation of $\widetilde{\mathcal{U}}_{\mathrm{top}}$]\label{thm:top-representation}
Let \(U\) be a uninorm on \(\mathbb{P}\) with neutral element
$e\in K\setminus\{0,1\}$. Then the following statements are equivalent:
\begin{enumerate}[label=\textup{(\arabic*)}, itemsep=1pt, parsep=0pt, topsep=8pt]
    \item \( U \in \widetilde{\mathcal{U}}_{\mathrm{top}} \).
    \item \( U \) can be constructed via the method described in Theorem~\ref{thm:top-construction}.
\end{enumerate}
Moreover, the three components are uniquely determined by $U$:
\[
\iota(x)=U(0,x)\quad(x\in \widehat B_e),\qquad
U_L=U|_{(A_e\cup\{e\})\times(A_e\cup\{e\})},\qquad
Q=U|_{\widehat B_e\times\widehat B_e}.
\]
\end{theorem}
\begin{proof}
\noindent\textbf{$(1 \Rightarrow 2)$}
Let \( U \in \widetilde{\mathcal{U}}_{\mathrm{top}} \).
Define the functions \(\iota,U_L,Q\) by
\[
\iota(x)=U(0,x)\quad(x\in \widehat B_e),\qquad
U_L=U|_{(A_e\cup\{e\})\times(A_e\cup\{e\})},\qquad
Q=U|_{\widehat B_e\times\widehat B_e}.
\]
Condition $(\circledast)$ and $0\in A_e$ show that
$\iota(d)=U(d,0)=q_d\in \widehat B_e$ for every $d\in \widehat B_e$. By
Lemma~\ref{lem:top-block-closure}, $U(A_e\times A_e)\subseteq A_e$; neutrality
then shows that $A_e\cup\{e\}$ is closed under $U$. Hence the
restriction $U_L$ is well defined on this set.
If $x\unlhd y$ in $\widehat B_e$, monotonicity gives
$\iota(x)=U(0,x)\unlhd U(0,y)=\iota(y)$, so $\iota$ is order-preserving.

First, we show that $\iota$ is an interior operator on $\widehat B_e$. Since $\iota$ is an order-preserving map on $\widehat B_e$, it remains to prove that
for \( x \in \widehat B_e \), we have
\begin{equation}\label{eq:top-interior-identities}
 \iota(x)\unlhd x, \ \text{and} \ \iota(\iota(x)) = \iota(x).
\end{equation}
We verify the two assertions in equation~\eqref{eq:top-interior-identities}
separately.
\begin{enumerate}[itemsep=1pt, parsep=0pt, topsep=8pt]
\item[(a)]   $\iota(x)\unlhd x$.

  Since $0\unlhd e$, monotonicity gives
  $\iota(x)=U(0,x)\unlhd U(e,x)=x$.
\item[(b)]  $\iota(\iota(x)) = \iota(x)$.

  Monotonicity and the fact that $0$ is the least element give
  $U(0,0)\unlhd U(e,0)=0$, hence $U(0,0)=0$. Therefore,
  associativity gives
\begin{align*}
\iota(\iota(x))
 &=U(0,\iota(x))=U(0,U(0,x))\\
 &=U(U(0,0),x)=U(0,x)\\
 &=\iota(x).
\end{align*}
\end{enumerate}

Second, we show that $U(\widehat B_e\times\widehat B_e)\subseteq \widehat B_e$. By monotonicity of $U$,
\(\iota(x)=U(x,0)\unlhd U(x,y)\) for $x,y\in \widehat B_e$.
Since $\iota(x)\in\widehat B_e$ and $\widehat B_e$ is an upper set by
Lemma~\ref{lem:top-block-closure}, we have
\(U(x,y)\in\widehat B_e\) for $x,y\in\widehat B_e$.
Thus $Q=U|_{\widehat B_e\times\widehat B_e}$ is indeed a map from
$\widehat B_e\times\widehat B_e$ into $\widehat B_e$.
Consequently, the restrictions $U_L$ and $Q$ inherit
monotonicity, commutativity, and associativity from $U$, and $U_L$
has neutral element $e$.
\rev{Thus $\iota$ is an interior operator, $U_L$ is a uninorm with
neutral element $e$, and $Q$ is an increasing, associative, and
commutative operation, as required in
Theorem~\ref{thm:top-construction}; it remains to verify conditions
\textup{(i)--(x)}.}


Third, we show that $U$ can be represented as in
equation~\eqref{eq:top-construction}.
By the definitions of the three restrictions, the following identities hold:
\begin{equation}\label{eq:top-lower-restriction}
U(x,y)=U_L(x,y),\ \text{for}\ (x,y)\in (A_e\cup\{e\})\times(A_e\cup\{e\}),
\end{equation}
\(U(x,y)=x,\ \text{for}\ (x,y)\in \widehat B_e \times\{e\}\),
\(U(x,y)=y,\ \text{for}\ (x,y)\in \{e\}\times\widehat B_e, \ \text{and}\)
\(U(x,y)=Q(x,y),\ \text{for}\ (x,y)\in \widehat B_e\times\widehat B_e\).
Moreover, by $(\circledast)$ in Definition~\ref{def:top-bottom-classes}, for $(x,y)\in A_e\times \widehat B_e$, we have
\begin{equation}\label{eq:top-mixed-recovery}
U(x,y)=U(0,y)=\iota(y).
\end{equation}
Similarly, for $(x,y)\in\widehat B_e\times A_e$, we obtain
\(U(x,y)=\iota(x)\).
Therefore, $U$ can be expressed as in
equation~\eqref{eq:top-construction}.

Fourth, we verify that condition~(i) in Theorem~\ref{thm:top-construction} holds.
It suffices to show that for $x,y\in A_e$ and $z\in \widehat B_e$ with
$x\unlhd z$, one has $U_L(x,y)\lhd \iota(z)$.
By monotonicity and commutativity of $U$,
$U_L(x,y)=U(y,x)\unlhd U(y,z)=\iota(z)$.
Since $U_L(x,y)\in A_e$, $\iota(z)\in \widehat B_e$, and
$A_e\cap\widehat B_e=\varnothing$ by
Lemma~\ref{lem:top-block-closure}, equality is impossible. Hence
\(U_L(x,y)\lhd \iota(z)\).

Fifth, we verify condition~(ii).
Let $x\in A_e$ and $y\in ]e,1]$. Since $(x,e)\preceq(x,y)$, monotonicity of $U$ gives
\(x=U(x,e)\unlhd U(x,y)=\iota(y)\).
Since $x\in A_e$ and $\iota(y)\in \widehat B_e$, equality is impossible by
Lemma~\ref{lem:top-block-closure}; therefore \(x\lhd \iota(y)\).

Sixth, we verify condition~(iii).
For $x,r\in A_e$ and $y,s\in \widehat B_e$ with $x\unlhd y$ and
$r\unlhd s$, we prove that $U_L(x,r)\unlhd Q(y,s)$.
By monotonicity of $U$, $U_L(x,r)=U(x,r)\unlhd U(y,s)=Q(y,s)$.

Seventh, we verify condition~(iv).
For $x\in A_e$, $y\in \widehat B_e$, and $s\in ]e,1]$ with $x\unlhd y$,
we show that $x\unlhd Q(y,s)$.
Indeed,
\(x=U(x,e)\unlhd U(y,s)=Q(y,s)\).

Eighth, we verify condition~(v).
For $x,y\in \widehat B_e$ with $x\unlhd y$, we prove that
\(\iota(x)\unlhd y\).
By monotonicity of $U$,
\(\iota(x)=U(x,0)\unlhd U(y,e)=y\).

Ninth, we verify condition~(vi).
For $x,y,s\in \widehat B_e$ with $x\unlhd y$, we obtain
\(\iota(x)\unlhd Q(y,s)\).
Indeed,
\(\iota(x)=U(x,0)\unlhd U(y,s)=Q(y,s)\).

Tenth, we verify condition~(vii).
For $x,y\in ]e,1]$, we show that
\(e\unlhd Q(x,y)\).
By monotonicity of $U$,
\(e=U(e,e)\unlhd U(x,y)=Q(x,y)\).

Eleventh, we verify condition~(viii).
For $x,y\in \widehat B_e$ and $s\in ]e,1]$ with $x\unlhd y$, we prove that
\(x\unlhd Q(y,s)\).
Indeed,
\(x=U(x,e)\unlhd U(y,s)=Q(y,s)\).

Twelfth, we verify condition~(ix).
For $x,y\in \widehat B_e$, we show that
\(\iota(Q(x,y))=Q(\iota(x),y)\).
By associativity of $U$,
\begin{align*}
\iota(Q(x,y))
 &=U(0,Q(x,y))=U(0,U(x,y))\\
 &=U(U(0,x),y)=U(\iota(x),y)\\
 &=Q(\iota(x),y).
\end{align*}

Finally, we verify condition~(x).
For $x\in [0,e[$, $y\in A_e$, and $z\in \widehat B_e$ with $y\unlhd z$, we prove that
\(U_L(x,y)\lhd z\).
By monotonicity of $U$,
\(U_L(x,y)=U(x,y)\unlhd U(e,z)=z\).
Lemma~\ref{lem:top-block-closure} gives $U_L(x,y)\in A_e$, whereas
$z\in \widehat B_e$; hence equality is impossible and $U_L(x,y)\lhd z$.
\rev{Therefore all the hypotheses and conditions \textup{(i)--(x)} of
Theorem~\ref{thm:top-construction} are satisfied.}

\par\medskip
\noindent\textbf{$(2 \Rightarrow 1)$}
The converse direction follows directly from Theorem~\ref{thm:top-construction}.

Finally, formula~\eqref{eq:top-construction} recovers $\iota$, $U_L$, and $Q$ by the
three displayed restrictions in the statement, proving uniqueness.
\end{proof}

The following example presents a non-trivial interior operator
and shows that the inclusion
$\widetilde{\mathcal{U}}_{\mathrm{top}}^{\star}\subseteq
\widetilde{\mathcal{U}}_{\mathrm{top}}$ can be proper.

\begin{example}
\label{ex:finite-utop-nonstar}
Let $Y=\{0,a,e,b,c,d,1\}$, and define the pseudo-order $\unlhd$ by
the following incidence matrix:
\[
\begin{array}{c|ccccccc}
\unlhd&0&a&e&b&c&d&1\\ \hline
0&1&1&1&1&1&1&1\\
a&0&1&1&1&0&1&1\\
e&0&0&1&1&0&1&1\\
b&0&0&0&1&1&1&1\\
c&0&0&0&0&1&1&1\\
d&0&0&0&0&0&1&1\\
1&0&0&0&0&0&0&1
\end{array}
\]
Its Hasse-type diagram is shown in Figure~\ref{fig:hasse-utop-nonstar}.

\begin{figure}[!htbp]
\centering
\includegraphics[width=0.30\textwidth]{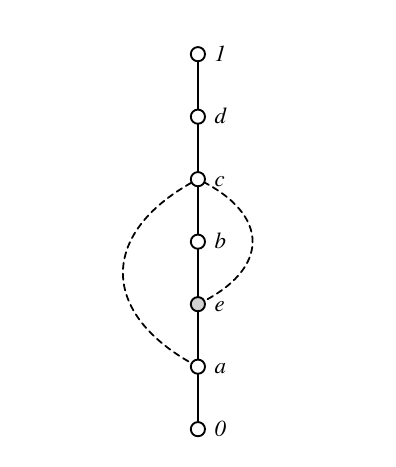}
\caption{\rev{Hasse-type diagram of $(Y,\unlhd)$ in
Example~\ref{ex:finite-utop-nonstar}.}}
\label{fig:hasse-utop-nonstar}
\end{figure}

The only incomparable pairs, apart from symmetry, are $(a,c)$ and
$(e,c)$. Their meets are $0$, while
$
a\vee c=e\vee c=d.
$
For every comparable pair, the smaller element is its meet and the
larger element is its join. Hence every pair has a meet and a join, and
$
\mathbb C_t=(Y,\unlhd,\wedge,\vee,0,1)
$
is a bounded trellis. It is a proper trellis, since
\[
e\unlhd b\unlhd c
\qquad\text{but}\qquad
e\ntrianglelefteq c.
\]
The element $e$ is middle-transitive and belongs to no non-trivial
cycle. Moreover,
\[
I_e^1=I_e^3=\varnothing,\qquad I_e^2=\{c\},\qquad
A_e=\{0,a\},\qquad \widehat B_e=\{b,c,d,1\}.
\]

Define $\iota:\widehat B_e\to \widehat B_e$ by
\[
\iota(b)=b,\qquad \iota(c)=b,\qquad \iota(d)=d,\qquad \iota(1)=1.
\]
Thus $\iota$ is a non-identity interior operator. Let
\[
U_L(x,y)=x\wedge y\qquad(x,y\in[0,e]=\{0,a,e\})
\]
and define $Q:\widehat B_e\times\widehat B_e\to \widehat B_e$ by
\[
Q(x,y)=
\begin{cases}
1, & \text{if }1\in\{x,y\},\\
d, & \text{otherwise}.
\end{cases}
\]
The operation $Q$ is non-constant, increasing, commutative, and
associative. The relation induced on $\widehat B_e$ is the chain
$b\lhd c\lhd d\lhd1$. Consequently, the interior-operator axioms for
$\iota$ and the algebraic properties of $Q$ follow directly from the two
displayed definitions.

We briefly indicate why the compatibility conditions of
Theorem~\ref{thm:top-construction} are non-vacuous here. One has
\[
\iota(\widehat B_e)=\{b,d,1\},\qquad Q(\widehat B_e\times\widehat B_e)=\{d,1\},
\]
and every element of $A_e$ is below $b,d$, and $1$. These observations
verify conditions \textup{(i)--(iv)}; moreover, $e\unlhd d$ and
$e\unlhd1$ verify condition \textup{(vii)}. Conditions
\textup{(v), (vi), and (viii)} follow from the displayed chain. For
condition \textup{(ix)}, both sides equal $1$ when one of the relevant
arguments is $1$, and otherwise both sides equal $d$. Finally, for
condition \textup{(x)}, if $y=0$, then $U_L(x,y)=0\lhd z$ for every
$z\in \widehat B_e$; if $y=a$, its antecedent forces
$z\in\{b,d,1\}$, and $U_L(x,a)\in\{0,a\}\lhd z$. Thus all ten
conditions of Theorem~\ref{thm:top-construction} hold.

Formula~\eqref{eq:top-construction} gives
\[
\begin{array}{c|ccccccc}
U&0&a&e&b&c&d&1\\ \hline
0&0&0&0&b&b&d&1\\
a&0&a&a&b&b&d&1\\
e&0&a&e&b&c&d&1\\
b&b&b&b&d&d&d&1\\
c&b&b&c&d&d&d&1\\
d&d&d&d&d&d&d&1\\
1&1&1&1&1&1&1&1
\end{array}
\]
Therefore $U\in\widetilde{\mathcal{U}}_{\mathrm{top}}$. In particular, for every
$x\in \widehat B_e$ and $y\in A_e$ one has $U(x,y)=\iota(x)$, which also displays
the recovery formula $\iota(x)=U(0,x)$ for a genuinely non-trivial interior
operator. On the other hand, \(c\in I_e^2\) and \(U(c,0)=b\ne c\).
Condition \textup{(A1)} in Definition~\ref{def:top-bottom-classes} fails.
By contrast, $\iota(y)=y$ for every $y\in]e,1]=\{b,d,1\}$, so
\textup{(A2)} holds with $m_y=y$, and \textup{(A3)} reduces to the identity
$x\vee y=x\vee y$. Thus \textup{(A1)} is the only one of
Conditions~\textup{(A1)--(A3)} that fails, and hence
$
U\notin\widetilde{\mathcal{U}}_{\mathrm{top}}^{\star}.
$
Together with
$\widetilde{\mathcal{U}}_{\mathrm{top}}^{\star}\subseteq
\widetilde{\mathcal{U}}_{\mathrm{top}}$, this proves that, on the bounded trellis
$\mathbb C_t$,
$
\widetilde{\mathcal{U}}_{\mathrm{top}}^{\star}
\subsetneq\widetilde{\mathcal{U}}_{\mathrm{top}}.
$
\end{example}

\subsection{The subclass
\texorpdfstring{$\widetilde{\mathcal{U}}_{\mathrm{top}}^{\star}$}{Utop-star}}

\rev{We next specialize the preceding representation to
$\widetilde{\mathcal{U}}_{\mathrm{top}}^{\star}$ by imposing conditions
\textup{(T1)} and \textup{(T2)} in addition to the hypotheses of
Theorem~\ref{thm:top-construction}.}

\begin{theorem}[Construction of $\widetilde{\mathcal{U}}_{\mathrm{top}}^{\star}$]
\label{th:utop-star-construction}
Let $\mathbb{P}=(X,\unlhd,\wedge,\vee,0,1)$ be a bounded trellis and
let $e\in K\setminus\{0,1\}$. Suppose that the maps $\iota$, $U_L$, and $Q$
satisfy all the hypotheses and conditions \textup{(i)--(x)} of
Theorem~\ref{thm:top-construction}. Assume, in addition, that
\begin{enumerate}[label=\textup{(T\arabic*)},leftmargin=*,itemsep=2pt]
\item $\iota(x)=x$ for every $x\in I_e^2\cup I_e^3$;
\item $x\vee y=x\vee \iota(y)$ for all
$x\in I_e^2\cup I_e^3$ and $y\in ]e,1]$.
\end{enumerate}
Then $\iota(]e,1])\subseteq]e,1]$, and the operation $U:X\times X\to X$ defined by
formula~\eqref{eq:top-construction} is a uninorm on $\mathbb{P}$ with
neutral element $e$. Moreover,
$U\in\widetilde{\mathcal{U}}_{\mathrm{top}}^{\star}$.
\end{theorem}

\begin{proof}
By Theorem~\ref{thm:top-construction}, $U$ is a uninorm with neutral element $e$ and
$U\in\widetilde{\mathcal{U}}_{\mathrm{top}}$. Let
$x\in I_e^2\cup I_e^3$ and $y\in A_e$. By
formula~\eqref{eq:top-construction} and condition~\textup{(T1)},
\(U(x,y)=\iota(x)=x\).
We next prove that
$\iota(]e,1])\subseteq]e,1]$, which will allow us to take
$m_y=\iota(y)$ in \textup{(A2)}.
Let $y\in]e,1]$. Since
$\iota(y)\in\widehat B_e$, an image outside $]e,1]$ would satisfy
$z=\iota(y)\in I_e^2\cup I_e^3$.
Since $z=\iota(y)$ and $\iota(y)\unlhd y$, we have
$z\unlhd y$, and hence $z\vee y=y$.
Condition~\textup{(T2)}, applied with $x=z$, gives
$
z\vee y=z\vee\iota(y)=z.
$
Together with $z\vee y=y$, this yields $z=y$, which is impossible because
$z\in I_e^2\cup I_e^3$ whereas $y\in]e,1]$.
Consequently, $\iota(]e,1])\subseteq]e,1]$. Put $m_y=\iota(y)$. Then
formula~\eqref{eq:top-construction} gives
\(U(y,a)=m_y\) for $a\in A_e$,
and condition~\textup{(T2)} gives
\(x\vee y=x\vee m_y\) for $x\in I_e^2\cup I_e^3$ and $y\in]e,1]$.
Thus conditions \textup{(A1)--(A3)} in Definition~\ref{def:top-bottom-classes} are
satisfied, and hence $U\in\widetilde{\mathcal{U}}_{\mathrm{top}}^{\star}$.
\end{proof}

\begin{theorem}[Representation of $\widetilde{\mathcal{U}}_{\mathrm{top}}^{\star}$]
\label{th:utop-star-representation}
Let $U$ be a uninorm on
$\mathbb{P}=(X,\unlhd,\wedge,\vee,0,1)$ with neutral element
$e\in K\setminus\{0,1\}$. Then the following statements are equivalent:
\begin{enumerate}[label=\textup{(\arabic*)},leftmargin=*,itemsep=2pt]
\item $U\in\widetilde{\mathcal{U}}_{\mathrm{top}}^{\star}$.
\item $U$ can be constructed via the method described in
Theorem~\ref{th:utop-star-construction}.
\end{enumerate}
Moreover, the components occurring in this construction are uniquely
determined by $U$ as follows:
\[
\iota(x)=U(0,x)\quad\bigl(x\in \widehat B_e\bigr),\qquad
U_L=U\big|_{(A_e\cup\{e\})\times(A_e\cup\{e\})},
\qquad
Q=U\big|_{\widehat B_e\times\widehat B_e}.
\]
\end{theorem}

\begin{proof}
$(1\Rightarrow 2)$ Let
$U\in\widetilde{\mathcal{U}}_{\mathrm{top}}^{\star}$. Since
$\widetilde{\mathcal{U}}_{\mathrm{top}}^{\star}\subseteq
\widetilde{\mathcal{U}}_{\mathrm{top}}$, Theorem~\ref{thm:top-representation} yields an interior
operator $\iota$, a uninorm $U_L$, and an increasing, commutative, and
associative operation $Q$ satisfying all the hypotheses and conditions
of Theorem~\ref{thm:top-construction}, such that $U$ is given by formula~\eqref{eq:top-construction}.

It remains to verify the two additional conditions in
Theorem~\ref{th:utop-star-construction}. If
$x\in I_e^2\cup I_e^3$, then condition \textup{(A1)} in
Definition~\ref{def:top-bottom-classes} and $0\in[0,e[$ give
\(\iota(x)=U(0,x)=U(x,0)=x\).
For $y\in]e,1]$, condition \textup{(A2)} gives $m_y\in]e,1]$ with
$\iota(y)=U(y,0)=m_y$. Condition \textup{(A3)} therefore yields
\(x\vee y=x\vee m_y=x\vee\iota(y)\) for $x\in I_e^2\cup I_e^3$ and $y\in]e,1]$.
Thus the assumptions of Theorem~\ref{th:utop-star-construction} hold, and
$U$ is obtained by that construction.

$(2\Rightarrow 1)$ This implication follows directly from
Theorem~\ref{th:utop-star-construction}.

The displayed formulas for $\iota$, $U_L$, and $Q$ follow immediately from
formula~\eqref{eq:top-construction}; therefore, no two different triples of components
can induce the same uninorm $U$.
\end{proof}

The following example illustrates
Theorem~\ref{th:utop-star-representation} on a proper bounded trellis.

\begin{example}
\label{ex:finite-utop-star}
Let $X=\{0,a,e,b,c,1\}$. Define a relation $\unlhd$ by the following
incidence matrix, where an entry $1$ means that the element indexing
the row is related to the element indexing the column:
\[
\begin{array}{c|cccccc}
\unlhd&0&a&e&b&c&1\\ \hline
0&1&1&1&1&1&1\\
a&0&1&1&1&0&1\\
e&0&0&1&1&0&1\\
b&0&0&0&1&1&1\\
c&0&0&0&0&1&1\\
1&0&0&0&0&0&1
\end{array}
\]
Figure~\ref{fig:hasse-utop-star} gives the corresponding Hasse-type diagram.

\begin{figure}[!htbp]
\centering
\includegraphics[width=0.30\textwidth]{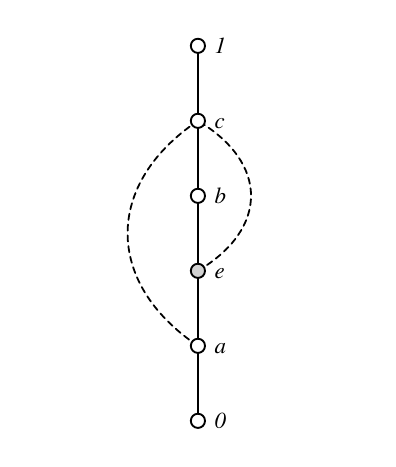}
\caption{Hasse-type diagram of the bounded trellis in
Example~\ref{ex:finite-utop-star}.}
\label{fig:hasse-utop-star}
\end{figure}

The only incomparable pairs, apart from symmetry, are $(a,c)$ and
$(e,c)$; both have meet $0$ and join $1$. A direct check of the matrix
shows that, for every comparable pair, the smaller element is its meet
and the larger element is its join. Thus every pair has a meet and a
join, and
$\mathbb{P}=(X,\unlhd,\wedge,\vee,0,1)$ is a bounded trellis. It is not a
lattice, because
\[
e\unlhd b\unlhd c
\qquad\text{but}\qquad
e\ntrianglelefteq c.
\]
The lower and upper relation sets of $e$ are
\[
\{x\in X:x\unlhd e\}=\{0,a,e\},
\qquad
\{y\in X:e\unlhd y\}=\{e,b,1\}.
\]
The matrix therefore shows that $e$ is middle-transitive. Since the
relation has no non-trivial directed cycle, $e\in K$. Moreover,
\[
I_e^1=I_e^3=\varnothing,\qquad I_e^2=\{c\},\qquad
A_e=\{0,a\},\qquad \widehat B_e=\{b,c,1\}.
\]

This example shows that an interior operator on $\widehat B_e$ need
not map $]e,1]$ into $]e,1]$. Define
$h:\widehat B_e\to \widehat B_e$ by
\[
h(b)=b,\qquad h(c)=c,\qquad h(1)=c.
\]
Since $b\unlhd c\unlhd1$, the map $h$ is order-preserving,
idempotent, and satisfies $h(x)\unlhd x$; hence it is an interior
operator on $\widehat B_e$. Nevertheless,
$
h(1)=c\notin]e,1]
$
because $c\parallel e$. Thus $h(]e,1])\subseteq]e,1]$ does not follow
from the interior-operator axioms on a trellis. Moreover, $h$ fails the
join identity since $c\vee1=1\ne c=c\vee h(1)$. This explains how the
join condition in Theorem~\ref{th:utop-star-construction} prevents such
an escape from the strict upper interval.

We now specify the maps in
Theorem~\ref{th:utop-star-construction} explicitly. Let $\iota$ be the identity map on
$\widehat B_e$, let
\[
U_L(x,y)=x\wedge y\qquad(x,y\in[0,e]=\{0,a,e\}),
\]
and let $Q(x,y)=1$ for all $x,y\in \widehat B_e$. Then $U_L$ is a uninorm on
$[0,e]$ with neutral element $e$, while $Q$ is increasing,
associative, and commutative. Here conditions \textup{(iii)}, \textup{(iv)},
\textup{(vi)--(ix)} of Theorem~\ref{thm:top-construction} follow from
$Q\equiv1$ and $\iota=\operatorname{id}$. Condition \textup{(v)} is immediate.
\rev{Condition \textup{(ii)} holds because $]e,1]=\{b,1\}$ and every element
of $A_e=\{0,a\}$ is strictly below both $b$ and $1$. For conditions
\textup{(i)} and \textup{(x)}, if $z\in\{b,1\}$, then
$U_L(x,y)\in\{0,a\}\lhd z$. If $z=c$, the antecedent in \textup{(i)} forces
$x=0$, whereas that in \textup{(x)} forces $y=0$; hence the required value
of $U_L$ is $0\lhd c$ in either case. Thus all ten conditions are
satisfied.} The two
additional conditions of Theorem~\ref{th:utop-star-construction} also
hold: $\iota(c)=c$ and
\[
c\vee b=c=c\vee \iota(b),\qquad
c\vee1=1=c\vee \iota(1).
\]
Formula~\eqref{eq:top-construction} therefore gives the following uninorm:
\[
\begin{array}{c|cccccc}
U&0&a&e&b&c&1\\ \hline
0&0&0&0&b&c&1\\
a&0&a&a&b&c&1\\
e&0&a&e&b&c&1\\
b&b&b&b&1&1&1\\
c&c&c&c&1&1&1\\
1&1&1&1&1&1&1
\end{array}
\]
For completeness, a direct finite check of this table gives
\[
U(x,U(y,z))=U(U(x,y),z)
\]
and
\[
x\unlhd y,\ r\unlhd s
\ \Rightarrow\ U(x,r)\unlhd U(y,s)
\qquad(x,y,r,s\in X).
\]
Furthermore, $U(c,0)=U(c,a)=c$, and for $x\in\{b,1\}$ one has
$U(x,0)=U(x,a)=x$. Hence conditions \textup{(A1)--(A3)} hold, and
$
U\in\widetilde{\mathcal{U}}_{\mathrm{top}}^{\star}.
$
The recovery formulas
\[
\iota(x)=U(0,x)\quad(x\in\widehat B_e),\qquad
U_L=U|_{\{0,a,e\}\times\{0,a,e\}},\qquad
Q=U|_{\{b,c,1\}\times\{b,c,1\}}
\]
also illustrate the uniqueness assertion in
Theorem~\ref{th:utop-star-representation}.
\end{example}

\subsection{Construction and representation of
\texorpdfstring{$\widetilde{\mathcal{U}}_{\mathrm{bot}}$}{Ubot}}


The results in this subsection are the dual counterparts of
those established for
$\widetilde{\mathcal{U}}_{\mathrm{top}}$. For completeness, we state
them explicitly but omit their proofs.
We denote the closure operator by $\kappa$ to avoid a conflict with
the element $g$ used in Part~II.

The following lemma is the counterpart of
Lemma~\ref{lem:top-block-closure}. It is stated explicitly for later
use in the representation of
$\widetilde{\mathcal{U}}_{\mathrm{bot}}$.

\begin{lemma}\label{lem:bottom-block-closure}
Let $e\in K\setminus\{0,1\}$. Then
the sets $\widehat A_e$, $\{e\}$ and $B_e$ form a partition of $X$.
 Moreover, $\widehat A_e$ is a lower set: if $d\in\widehat A_e$ and
$z\unlhd d$, then $z\in\widehat A_e$. Finally, if $V:X\times X\to X$ is an
increasing operation with neutral element $e$, then
$V(B_e\times B_e)\subseteq B_e$.
\end{lemma}


\begin{theorem}[Construction of $\widetilde{\mathcal{U}}_{\mathrm{bot}}$]
\label{th:ubot-construction}
Let $\mathbb{P}=(X,\unlhd,\wedge,\vee,0,1)$ be a bounded trellis, and let
$e\in K\setminus\{0,1\}$. Consider a closure operator
$\kappa:\widehat A_e\to\widehat A_e$, a uninorm
$U_R:(B_e\cup\{e\})\times(B_e\cup\{e\})\to B_e\cup\{e\}$ with neutral element $e$, and an
increasing, associative, and commutative operation
$P:\widehat A_e\times\widehat A_e\to\widehat A_e$. Assume that the following conditions
hold:
\begin{enumerate}[itemsep=1pt, parsep=0pt, topsep=8pt]
\item[\textup{(i)}]
If $x,y\in B_e$ and $z\in\widehat A_e$ satisfy $z\unlhd x$, then
$\kappa(z)\lhd U_R(x,y)$.

\item[\textup{(ii)}]
If $x\in[0,e[$ and $y\in B_e$, then $\kappa(x)\lhd y$.

\item[\textup{(iii)}]
If $x,r\in\widehat A_e$ and $y,s\in B_e$ satisfy
$x\unlhd y$ and $r\unlhd s$, then
$P(x,r)\unlhd U_R(y,s)$.

\item[\textup{(iv)}]
If $x\in\widehat A_e$, $y\in B_e$, and $s\in[0,e[$ satisfy
$x\unlhd y$, then $P(x,s)\unlhd y$.

\item[\textup{(v)}]
If $x,y\in\widehat A_e$ satisfy $x\unlhd y$, then $x\unlhd\kappa(y)$.

\item[\textup{(vi)}]
If $x,y,s\in\widehat A_e$ satisfy $x\unlhd y$, then
$P(x,s)\unlhd\kappa(y)$.

\item[\textup{(vii)}]
For all $x,y\in[0,e[$, $P(x,y)\unlhd e$.

\item[\textup{(viii)}]
If $x,y\in\widehat A_e$ and $s\in[0,e[$ satisfy $x\unlhd y$, then
$P(x,s)\unlhd y$.

\item[\textup{(ix)}]
For all $x,y\in\widehat A_e$,
$\kappa(P(x,y))=P(\kappa(x),y)$.

\item[\textup{(x)}]
If $x\in]e,1]$, $y\in B_e$, and $z\in\widehat A_e$ satisfy
$z\unlhd y$, then $z\lhd U_R(x,y)$.
\end{enumerate}
The absence of an extra hypothesis $x\unlhd y$ in \textup{(ii)} is
essential: $(x,y)\preceq(e,y)$ for every $x\in[0,e[$ and
$y\in B_e$. Define $U:X\times X\to X$ by
\begin{equation}\label{eq:ubot-construction}
U(x,y)=
\begin{cases}
U_R(x,y),
   & (x,y)\in(B_e\cup\{e\})\times(B_e\cup\{e\}),\\
x,
   & (x,y)\in\widehat A_e\times\{e\},\\
y,
   & (x,y)\in\{e\}\times\widehat A_e,\\
\kappa(x),
   & (x,y)\in\widehat A_e\times B_e,\\
\kappa(y),
   & (x,y)\in B_e\times\widehat A_e,\\
P(x,y),
   & (x,y)\in\widehat A_e\times\widehat A_e.
\end{cases}
\end{equation}
Then $U$ is a uninorm on $\mathbb{P}$ with neutral element $e$.
Moreover, $U(x,y)=\kappa(x)\in\widehat A_e$ for all
$x\in\widehat A_e$ and $y\in B_e$, and hence
$U\in\widetilde{\mathcal{U}}_{\mathrm{bot}}$.
\end{theorem}


\begin{theorem}[Representation of $\widetilde{\mathcal{U}}_{\mathrm{bot}}$]
\label{th:ubot-representation}
Let $U$ be a uninorm on
$\mathbb{P}=(X,\unlhd,\wedge,\vee,0,1)$ with neutral element
$e\in K\setminus\{0,1\}$. Then the following statements are equivalent:
\begin{enumerate}[label=\textup{(\arabic*)}, itemsep=1pt, parsep=0pt, topsep=8pt]
\item $U\in\widetilde{\mathcal{U}}_{\mathrm{bot}}$.
\item There exist a closure operator
$\kappa:\widehat A_e\to\widehat A_e$, a uninorm
$U_R:(B_e\cup\{e\})\times(B_e\cup\{e\})\to B_e\cup\{e\}$ with neutral element $e$, and an
increasing, associative, and commutative operation
$P:\widehat A_e\times\widehat A_e\to\widehat A_e$ satisfying conditions \textup{(i)--(x)} of
Theorem~\ref{th:ubot-construction}, such that $U$ is given by
formula~\eqref{eq:ubot-construction}.
\end{enumerate}
Moreover, these components are uniquely determined by $U$:
\[
\kappa(x)=U(x,1)\quad(x\in\widehat A_e),\qquad
U_R=U|_{(B_e\cup\{e\})\times(B_e\cup\{e\})},\qquad
P=U|_{\widehat A_e\times\widehat A_e}.
\]
\end{theorem}


\subsection{The subclass
\texorpdfstring{$\widetilde{\mathcal{U}}_{\mathrm{bot}}^{\star}$}{Ubot-star}}

We next specialize the preceding representation to
$\widetilde{\mathcal{U}}_{\mathrm{bot}}^{\star}$. In addition to the
conditions of Theorem~\ref{th:ubot-construction}, we impose the two
further requirements stated below. The second requirement implies
$\kappa([0,e[)\subseteq[0,e[$, and hence this inclusion need not be
assumed separately.

\begin{theorem}[Construction of $\widetilde{\mathcal{U}}_{\mathrm{bot}}^{\star}$]
\label{th:ubot-star-construction}
Let $\mathbb{P}=(X,\unlhd,\wedge,\vee,0,1)$ be a bounded trellis and
let $e\in K\setminus\{0,1\}$. Suppose that $\kappa$, $U_R$, and $P$
satisfy all the hypotheses and conditions \textup{(i)--(x)} of
Theorem~\ref{th:ubot-construction}. Assume, in addition, that
\begin{enumerate}[label=\textup{(D\arabic*)},leftmargin=*,itemsep=2pt]
\item $\kappa(x)=x$ for every $x\in I_e^1\cup I_e^3$;
\item $x\wedge y=x\wedge\kappa(y)$ for all
$x\in I_e^1\cup I_e^3$ and $y\in[0,e[$.
\end{enumerate}
Then $\kappa([0,e[)\subseteq[0,e[$, and the operation $U:X\times X\to X$ defined by
formula~\eqref{eq:ubot-construction} is a uninorm on $\mathbb{P}$ with
neutral element $e$, and $U\in\widetilde{\mathcal{U}}_{\mathrm{bot}}^{\star}$.
\end{theorem}


\begin{theorem}[Representation of
$\widetilde{\mathcal{U}}_{\mathrm{bot}}^{\star}$]
\label{th:ubot-star-representation}
Let $U$ be a uninorm on
$\mathbb{P}=(X,\unlhd,\wedge,\vee,0,1)$ with neutral element
$e\in K\setminus\{0,1\}$. Then the following statements are equivalent:
\begin{enumerate}[label=\textup{(\arabic*)},leftmargin=*,itemsep=2pt]
\item $U\in\widetilde{\mathcal{U}}_{\mathrm{bot}}^{\star}$.
\item $U$ can be constructed by the method of
Theorem~\ref{th:ubot-star-construction}.
\end{enumerate}
Moreover, its components are uniquely determined by
\[
\kappa(x)=U(x,1)\quad(x\in\widehat A_e),\qquad
U_R=U|_{(B_e\cup\{e\})\times(B_e\cup\{e\})},\qquad
P=U|_{\widehat A_e\times\widehat A_e}.
\]
\end{theorem}


\section{Part II: \texorpdfstring{Representations}{Representations} of the classes
\texorpdfstring{$\widetilde{\mathcal{U}}_{\mathrm{max}}^{r}$ and $\widetilde{\mathcal{U}}_{\mathrm{min}}^{r}$}
{Umax-r and Umin-r}}
\label{sec:r-structural-properties}

In this part, we investigate the classes
$\widetilde{\mathcal{U}}_{\mathrm{max}}^{r}$ and
$\widetilde{\mathcal{U}}_{\mathrm{min}}^{r}$. On each prescribed mixed
region, their value is one of the two arguments rather than the image of a
unary operator. Our purpose is to obtain an exact piecewise description with
unique recovery of its components.
We first determine the regions on which the operation is forced.
For $U\in\widetilde{\mathcal{U}}_{\mathrm{max}}^{r}$, we identify the
pairs in $I_e^3\times I_e^3$ whose values lie in $A_e$ or are equal to
$e$; the remaining unprescribed values on
$\widehat B_e\times\widehat B_e$ are described by a partial operation.
Dually, for
$U\in\widetilde{\mathcal{U}}_{\mathrm{min}}^{r}$, we identify the
pairs in $I_e^3\times I_e^3$ whose values lie in $B_e$ or are equal to
$e$; the remaining unprescribed values on
$\widehat A_e\times\widehat A_e$ are described by a partial operation.
We then state the conditions under which these pieces assemble into a
uninorm. These conditions organize the verification of monotonicity and
associativity according to the different cases in the piecewise definition;
they are not intended to weaken the corresponding axioms for the resulting
operation.


%
%

\subsection{Definitions and structural setting}

We first adapt the classes
 introduced for bounded
pseudo-ordered sets in~\cite{KL26} to the common regions defined in
Section~\ref{sec:preliminaries}.

\begin{definition}\label{def:umaxr-uminr-trellis}
Let $\mathbb{P}=(X,\unlhd,\wedge,\vee,0,1)$ be a bounded trellis and
let $e\in K\setminus\{0,1\}$.
\begin{enumerate}[label=\textup{(\roman*)}, itemsep=2pt,
                  parsep=0pt, topsep=4pt]
  \item We define $\widetilde{\mathcal{U}}_{\mathrm{max}}^{r}$ as the class of
  all uninorms $U$ with neutral element $e$ such that
  \[
  U(x,y)=x\qquad ((x,y)\in A_e\times\widehat B_e).
  \]

  \item We define $\widetilde{\mathcal{U}}_{\mathrm{min}}^{r}$ as the class of
  all uninorms $U$ with neutral element $e$ such that
  \[
  U(x,y)=x\qquad ((x,y)\in B_e\times\widehat A_e).
  \]
\end{enumerate}
\end{definition}

By commutativity, the first identity is equivalently $U(x,y)=y$ on
$\widehat B_e\times A_e$, while the second is equivalently $U(x,y)=y$ on
$\widehat A_e\times B_e$. If $\unlhd$ is transitive, then
$I_e^1=I_e^2=\varnothing$ and $I_e^3=I_e$; Definition~\ref{def:umaxr-uminr-trellis}
therefore gives the corresponding classes
$\mathcal{U}_{\mathrm{max}}^{r}$ and
$\mathcal{U}_{\mathrm{min}}^{r}$ in~\cite{HP21}.

The distinction among $I_e^1$, $I_e^2$, and $I_e^3$ is essential.
On a proper trellis, the existence of a finite sequence
$x=x_0\unlhd x_1\unlhd\cdots\unlhd x_n=y$ does not necessarily imply
$x\unlhd y$. Therefore, the representations below treat $I_e^3$,
rather than the whole set $I_e$, separately.

Throughout Part~II, a subset $D$ of $I_e^3\times I_e^3$ is called a lower set if
\[
 \boldsymbol{x},\boldsymbol{y}\in I_e^3\times I_e^3,\quad
 \boldsymbol{x}\in D,\quad
 \boldsymbol{y}\preceq\boldsymbol{x}
 \quad\Rightarrow\quad \boldsymbol{y}\in D.
\]
Likewise, a subset $D$ is an upper set if
\[
 \boldsymbol{x},\boldsymbol{y}\in I_e^3\times I_e^3,
 \quad \boldsymbol{x}\in D,
 \quad \boldsymbol{x}\preceq\boldsymbol{y}
 \quad\Rightarrow\quad \boldsymbol{y}\in D.
\]
For any $D\subseteq I_e^3\times I_e^3$, put
\[
 \operatorname{supp}(D)=
 \{x\in I_e^3\mid (x,y)\in D\text{ for some }y\in I_e^3\}.
\]

 A symmetric set $E\subseteq I_e^3\times I_e^3$ is said to satisfy the
\emph{unique-partner condition} if
\[
 (x,y),(x,z)\in E\quad\Rightarrow\quad y=z
 \qquad(x,y,z\in I_e^3).
\]
Equivalently, each $x\in\operatorname{supp}(E)$ occurs with a unique
$y\in I_e^3$ such that $(x,y)\in E$.

\subsection{Basic properties and existence conditions}

We first establish several properties and existence conditions that will
be used in the subsequent representations.

\begin{proposition}\label{prop:r-cone-closure}
Let $U$ be a uninorm on $\mathbb{P}$ with neutral element $e$. Then
\[
 U(A_e\times A_e)\subseteq A_e,
 \qquad
 U(B_e\times B_e)\subseteq B_e.
\]
More precisely,
\[
 U(x,y)\lesssim x\ \text{and}\ U(x,y)\lesssim y
 \quad (x,y\in A_e),
 \qquad
 x\lesssim U(x,y)\ \text{and}\ y\lesssim U(x,y)
 \quad (x,y\in B_e).
\]
Furthermore, $U|_{[0,e]\times[0,e]}$ is a t-norm on $[0,e]$ and
$U|_{[e,1]\times[e,1]}$ is a t-conorm on $[e,1]$. In particular,
\[
 U(x,y)\unlhd x\wedge y \quad(x,y\in[0,e]),
 \qquad
 x\vee y\unlhd U(x,y) \quad(x,y\in[e,1]).
\]
\end{proposition}

\begin{proof}
For $x,y\in A_e$, Lemma~\ref{lem:r-reachability} and neutrality give
\(U(x,y)\lesssim U(x,e)=x\) and \(U(x,y)\lesssim U(e,y)=y\).
Consequently, $U(x,y)\lesssim e$. If $U(x,y)=e$, then
$e\lesssim x\lesssim e$, which produces a non-trivial cycle containing $e$.
Hence $U(x,y)\in A_e$.

For $x,y\in B_e$, Lemma~\ref{lem:r-reachability} and neutrality give
\(x=U(x,e)\lesssim U(x,y)\) and \(y=U(e,y)\lesssim U(x,y)\).
Consequently, $e\lesssim U(x,y)$. If $U(x,y)=e$, then
$e\lesssim x\lesssim e$, again producing a non-trivial cycle containing
$e$. Hence $U(x,y)\in B_e$.

If $x,y\in[0,e]$, then $(x,y)\preceq(e,e)$, and monotonicity yields
$U(x,y)\unlhd U(e,e)=e$; hence $[0,e]$ is closed under $U$. Moreover,
$U(x,y)\unlhd U(x,e)=x$ and, symmetrically, $U(x,y)\unlhd y$. Thus the
restriction has neutral element $e$ and its value is below the meet of its
arguments. It inherits commutativity, associativity, and monotonicity from
$U$, and hence is a t-norm on $[0,e]$.

If $x,y\in[e,1]$, monotonicity gives
$e=U(e,e)\unlhd U(x,y)$, so $[e,1]$ is closed under $U$. Moreover,
$x=U(x,e)\unlhd U(x,y)$ and $y=U(e,y)\unlhd U(x,y)$; hence
$x\vee y\unlhd U(x,y)$. The restriction inherits commutativity,
associativity, and monotonicity from $U$ and has neutral element $e$.
\end{proof}

The following proposition identifies the regions determined by the
condition $U\in\widetilde{\mathcal{U}}_{\mathrm{max}}^{r}$.

\begin{proposition}\label{prop:r-max-blocks}
Let $U\in\widetilde{\mathcal{U}}_{\mathrm{max}}^{r}$. Then the following statements hold.
\begin{enumerate}[label=\textup{(\roman*)}, itemsep=2pt]
 \item $U(a,x)\in A_e$ for every $a\in A_e$ and $x\in X$, and
       $U(a,x)=a$ whenever $a\in A_e$ and
       $x\in\{e\}\cup\widehat B_e$.
 \item The element $0$ is an annihilator of $U$, that is,
       $U(0,x)=0$ for all $x\in X$.
 \item $U(B_e\times B_e)\subseteq B_e$ and
       $U(B_e\times I_e^3)\subseteq\widehat B_e$.
       Consequently, if $x,y\in\widehat B_e$ and
       $U(x,y)\in A_e\cup\{e\}$, then $x,y\in I_e^3$.
 \item For every $a\in A_e$ and $d\in\widehat B_e$, we have
       $a\lesssim d$ and $d\not\lesssim a$.
       If $a\in[0,e[$, then the relation $a\lhd d$ holds.
\end{enumerate}
\end{proposition}

\begin{proof}
Part~(i) follows from Proposition~\ref{prop:r-cone-closure}, neutrality, and
Definition~\ref{def:umaxr-uminr-trellis}. To prove (ii), it remains only to
consider $x\in A_e$. Since $x\lesssim e$, Lemma~\ref{lem:r-reachability}
gives $U(0,x)\lesssim U(0,e)=0$. In a bounded psoset, a path ending at the
smallest element $0$ can contain only $0$; hence $U(0,x)=0$.

The first inclusion in \textup{(iii)} follows from
Proposition~\ref{prop:r-cone-closure}. To prove the second inclusion,
let $b\in B_e$ and $c\in I_e^3$. Then
$
c=U(e,c)\lesssim U(b,c).
$
If $U(b,c)\in A_e\cup\{e\}$, then $U(b,c)\lesssim e$, and hence
$c\lesssim e$, contradicting $c\in I_e^3$. Therefore,
$
U(b,c)\notin A_e\cup\{e\},
$
so $U(b,c)\in\widehat B_e$.
This proves the second inclusion and the concluding assertion in
part~\textup{(iii)}.


Finally, for $a\in A_e$ and $d\in\widehat B_e$, monotonicity along a path
from $a$ to $e$ yields $a=U(a,d)\lesssim U(e,d)=d$.
A directed path from $d$ to $a$ would imply either a cycle containing $e$ or
$d\lesssim e$, according as $d\in B_e$ or $d\in I_e^3$. Both are impossible.
If $a\unlhd e$, the displayed path argument consists of one direct
inequality and gives $a\unlhd d$; the inequality is strict because $A_e$ and $\widehat B_e$
are disjoint.
\end{proof}

The corresponding nonemptiness criteria are already known for bounded
psosets. We recall them here because they provide the initial
constraints for the constructions developed below.

\begin{proposition}[{\cite[Remarks 7 and 8]{KL26}}]
\label{prop:r-nonempty}
The following equivalences hold:
\begin{align*}
 \widetilde{\mathcal{U}}_{\mathrm{max}}^{r}\ne\varnothing
 &\quad\Leftrightarrow\quad
 x\lhd y
 \quad\text{for every }(x,y)\in]0,e[\times(I_e^2\cup I_e^3),\\
 \widetilde{\mathcal{U}}_{\mathrm{min}}^{r}\ne\varnothing
 &\quad\Leftrightarrow\quad
 x\lhd y
 \quad\text{for every }(x,y)\in(I_e^1\cup I_e^3)\times]e,1[.
\end{align*}
\end{proposition}

\begin{proof}
Every bounded trellis is, in particular, a bounded psoset. Moreover,
when the notation of~\cite{KL26} is adopted, the commutativity of
uninorms implies that the classes
$\widetilde{\mathcal{U}}_{\mathrm{max}}^{r}$ and
$\widetilde{\mathcal{U}}_{\mathrm{min}}^{r}$ in
Definition~\ref{def:umaxr-uminr-trellis} coincide, respectively, with
the classes $\mathcal{U}_{\mathrm{max}r}^{\mathrm{ps}}$ and
$\mathcal{U}_{\mathrm{min}r}^{\mathrm{ps}}$ introduced for bounded
psosets in~\cite{KL26}. Therefore, the first and second equivalences
follow from Remarks~8\textup{(i)} and~7\textup{(i)}
of~\cite{KL26}, respectively.
\end{proof}


Part~\textup{(iv)} of Proposition~\ref{prop:r-max-blocks} provides the
comparison needed for $[0,e[$, while the second equivalence in
Proposition~\ref{prop:r-nonempty} provides the dual comparison for
$]e,1]$. Together, they yield the following conclusions concerning the
greatest and smallest elements of these intervals.

\begin{corollary}\label{cor:r-boundary-elements}
The following statements hold:
\begin{enumerate}
\item[\textup{(i)}]
If $\widetilde{\mathcal{U}}_{\mathrm{max}}^{r}\ne\varnothing$ and
$I_e^2\cup I_e^3\ne\varnothing$, then $[0,e[$ has a greatest element
$g$, and
\[
g=e\wedge c
\qquad\text{for every }c\in I_e^2\cup I_e^3.
\]

\item[\textup{(ii)}]
If $\widetilde{\mathcal{U}}_{\mathrm{min}}^{r}\ne\varnothing$ and
$I_e^1\cup I_e^3\ne\varnothing$, then $]e,1]$ has a smallest element
$s$, and
\[
s=e\vee c
\qquad\text{for every }c\in I_e^1\cup I_e^3.
\]
\end{enumerate}
In each case, the extremal element is independent of the choice of $c$.
\end{corollary}

\begin{proof}
For (i), choose $U\in\widetilde{\mathcal{U}}_{\mathrm{max}}^{r}$ and fix
$c\in I_e^2\cup I_e^3$. Every $x\in[0,e[$ satisfies $x\unlhd c$ by
Proposition~\ref{prop:r-max-blocks}. Thus $x$ is a lower bound of $e$ and $c$,
and $x\unlhd e\wedge c$. Since $c\parallel e$, one has
$e\wedge c\lhd e$. Hence $e\wedge c$ is the greatest element of $[0,e[$.
Its uniqueness makes it independent of $c$.

For (ii), fix $c\in I_e^1\cup I_e^3$. By
Proposition~\ref{prop:r-nonempty}, one has $c\lhd y$ for every
$y\in]e,1[$; the inequality $c\unlhd1$ holds because $1$ is the greatest
element. Hence every $y\in]e,1]$ is an upper bound of $e$ and $c$, and
therefore $e\vee c\unlhd y$. Since $c\parallel e$, one has
$e\lhd e\vee c$. Thus $e\vee c$ is the smallest element of $]e,1]$.
Its uniqueness makes it independent of $c$.
\end{proof}

We shall also use the following estimates for the direct relation $\unlhd$,
rather than only for  the   relation $\lesssim$.

\begin{proposition}\label{prop:r-order-estimates}
Let $U\in\widetilde{\mathcal{U}}_{\mathrm{max}}^{r}$ and put $J_e^+=I_e^2\cup I_e^3$. Then
\begin{align*}
 (x\wedge e)\vee(y\wedge e)&\unlhd U(x,y)
     &&(x,y\in J_e^+),\\
 x&\unlhd U(x,y)
     &&(x\in J_e^+,\ y\in]e,1]),\\
 x\vee y&\unlhd U(x,y)
     &&(x,y\in]e,1]).
\end{align*}
If $I_e^3\ne\varnothing$ and $g$ is the element in
Corollary~\ref{cor:r-boundary-elements}, then $g\unlhd U(x,y)$ for all
$x,y\in I_e^3$.
\end{proposition}

\begin{proof}
For $x\in J_e^+$, the element $x\wedge e$ belongs to $[0,e[$. Therefore, for
$x,y\in J_e^+$,
\[x\wedge e=U(x\wedge e,y)\unlhd U(x,y).\]
The analogous inequality holds for $y\wedge e$, so their join is also below
$U(x,y)$. If $y\in]e,1]$, then
$x=U(x,e)\unlhd U(x,y)$. The third inequality follows from
Proposition~\ref{prop:r-cone-closure}. Finally, $g\unlhd x$ for every
$x\in I_e^3$ and $U(g,y)=g$; monotonicity gives
$g\unlhd U(x,y)$.
\end{proof}

\subsection{
Values on
\texorpdfstring{$I_e^3\times I_e^3$}{Ie3 x Ie3} for
\texorpdfstring{$\widetilde{\mathcal{U}}_{\mathrm{max}}^{r}$}{Umax-r}}

Let $U\in\widetilde{\mathcal{U}}_{\mathrm{max}}^{r}$.
Proposition~\ref{prop:r-max-blocks} shows that
$
U(B_e\times B_e)\subseteq B_e
\ \text{and}\
U(B_e\times I_e^3)\subseteq\widehat B_e.
$
By the commutativity of $U$, the second inclusion also gives
$
U(I_e^3\times B_e)\subseteq\widehat B_e.
$
Consequently, if $(x,y)\in\widehat B_e\times\widehat B_e$ and at least one of
$x$ and $y$ belongs to $B_e$, then $U(x,y)\in\widehat B_e$.
Hence, a value of $U$ on $\widehat B_e\times\widehat B_e$ can leave
$\widehat B_e$ only when both arguments belong to $I_e^3$.

 Since $X=A_e\mathbin{\cup}\{e\}
\mathbin{\cup}\widehat B_e$, every such value belongs to
$A_e\cup\{e\}$. Hence, if $(x,y)\in I_e^3\times I_e^3$ and
$U(x,y)\notin\widehat B_e$, then either $U(x,y)\in A_e$ or
$U(x,y)=e$. This motivates the following definitions:
\[
\Delta_U^-=
\{(x,y)\in I_e^3\times I_e^3\mid U(x,y)\in A_e\},
\qquad
E_U^-=
\{(x,y)\in I_e^3\times I_e^3\mid U(x,y)=e\}.
\]
Thus, $\Delta_U^-$ consists of the pairs
$(x,y)\in I_e^3\times I_e^3$ for which $U(x,y)\in A_e$, whereas
$E_U^-$ consists of those for which $U(x,y)=e$.

\begin{theorem}\label{thm:r-max-exceptional}
Let $U\in\widetilde{\mathcal{U}}_{\mathrm{max}}^{r}$. Then the following statements
hold.
\begin{enumerate}[label=\textup{(\roman*)}, itemsep=2pt]
 \item $\Delta_U^-$ and $E_U^-$ are symmetric lower sets in
       $(I_e^3\times I_e^3,\preceq)$ and $\Delta_U^-\cap E_U^-=\varnothing$.
 \item If $\Delta_U^-\ne\varnothing$, then all its pairs have the same value
       $g_U\in[0,e[$. Moreover,
       \[
       U(a,g_U)=a\qquad(a\in A_e),
       \]
       $g_U$ is the greatest element of $A_e$, and
       \[
       g_U=e\wedge c\qquad(c\in I_e^3).
       \]
 \item For every $x\in I_e^3$, there is at most one $y\in I_e^3$ such that
       $(x,y)\in E_U^-$. Hence the symmetric set $E_U^-$ satisfies the
       unique-partner condition, and its support is
       \[
       C_{U,0}^-=
       \{x\in I_e^3\mid U(x,y)=e\text{ for some }y\in I_e^3\}.
       \]
       Every element of $C_{U,0}^-$ is minimal in $(I_e^3,\unlhd)$, and
       distinct elements of $C_{U,0}^-$ are incomparable.
 \item Put
       \[
       C_{U,-}^-=
       \{x\in I_e^3\mid U(x,y)\in A_e
                    \text{ for some }y\in I_e^3\}.
       \]
       Then $C_{U,-}^-$ is a lower set in $I_e^3$,
       $C_{U,-}^-\cap C_{U,0}^-=\varnothing$, and every element of
       $C_{U,-}^-$ is incomparable with every element of $C_{U,0}^-$.
\end{enumerate}
\end{theorem}

\begin{proof}
The commutativity of $U$ shows that both sets are symmetric. Moreover,
since $e\notin A_e$, no pair can belong to both $\Delta_U^-$ and $E_U^-$;
hence $\Delta_U^-\cap E_U^-=\varnothing$.
Assume first that $(x,y)\in\Delta_U^-$ and write $U(x,y)=h\in A_e$.
For every $a\in A_e$, associativity and the forced mixed values give
\[
 U(a,h)=U(a,U(x,y))=U(U(a,x),y)=U(a,y)=a.                \tag{\(*\)}
\]
If another pair in $\Delta_U^-$ has value $k$, the analogue of~$(*)$ gives
$U(a,k)=a$ for every $a\in A_e$. Taking $a=k$ in~$(*)$ gives
$U(k,h)=k$, whereas taking $a=h$ in the corresponding identity for $k$
gives $U(h,k)=h$. Hence commutativity yields $h=k$. Denote this common
value by $g_U$. Since $a\unlhd 1$ for every
$a\in A_e$, formula~$(*)$ and monotonicity yield
\(a=U(g_U,a)\unlhd U(g_U,1)=g_U\).
Thus $g_U$ is the greatest element of $A_e$.

It remains to see that $g_U$ is strictly below $e$ in the original
pseudo-order. Suppose that
$g_U\in I_e^1$ and choose a shortest path
$g_U=x_0\unlhd x_1\unlhd\cdots\unlhd x_n=e$. Then $n\geq2$, and each
$x_i$ with $i<n$ belongs to $A_e$. The maximality of $g_U$ gives
$x_1\unlhd g_U$, while the first edge gives $g_U\unlhd x_1$.
Antisymmetry yields $x_1=g_U$, contradicting the choice of a shortest path.
Thus $g_U\in[0,e[$, and Corollary~\ref{cor:r-boundary-elements} gives
$g_U=e\wedge c$ for every $c\in I_e^3$. This proves (ii).

We next prove that $\Delta_U^-$ is a lower set. Suppose
$(p,q)\preceq(x,y)$ with $p,q,x,y\in I_e^3$ and
$(x,y)\in\Delta_U^-$. Then $U(p,q)\unlhd g_U$, and consequently
$U(p,q)\lesssim e$. Thus $U(p,q)\in A_e\cup\{e\}$. The second possibility
would give $e\unlhd g_U\lhd e$, contrary to antisymmetry. Hence
$(p,q)\in\Delta_U^-$. This completes (i) for $\Delta_U^-$.

Suppose $U(x,y)=e=U(x,z)$ with $x,y,z\in I_e^3$. Associativity gives
\[
 y=U(y,e)=U(y,U(x,z))=U(U(y,x),z)=U(e,z)=z,
\]
which proves uniqueness of the partner and hence the assertion about $E_U^-$.

Let $U(x,y)=e$ and suppose that $x'\lhd x$ with $x'\in I_e^3$. Put
$w=U(x',y)$. Monotonicity gives $w\unlhd e$. If $w=e$, uniqueness of the
partner of $y$ gives $x'=x$. Otherwise $w\in[0,e[$, and associativity yields
\[
 U(w,x)=U(U(x',y),x)=U(x',U(y,x))=U(x',e)=x'.
\]
The forced mixed value gives $U(w,x)=w$, which is impossible because
$w\in[0,e[$ and $x'\in I_e^3$. Thus $x$ is minimal in $I_e^3$. It follows that
the support $C_{U,0}^-$ is an antichain and that $E_U^-$ is a lower set. This
proves (iii) and the remaining assertion in (i).

Finally, the lower-set property of $C_{U,-}^-$ follows from that of
$\Delta_U^-$. If an element $x$ belonged to both supports, choose
$y,z\in I_e^3$ such that $U(x,y)=g_U$ and $U(x,z)=e$. Then
\[
 y=U(y,U(x,z))=U(U(y,x),z)=U(g_U,z)=g_U,
\]
a contradiction. If $p\in C_{U,0}^-$ and $q\in C_{U,-}^-$ were
comparable, then either $p\unlhd q$ or $q\unlhd p$. In the first case,
the lower-set property of $C_{U,-}^-$ would give $p\in C_{U,-}^-$; in the
second, the minimality of $p$ would give $q=p$. Both conclusions contradict
$C_{U,-}^-\cap C_{U,0}^-=\varnothing$.{} This proves (iv).
\end{proof}

After removing $\Delta_U^-$ and $E_U^-$, define the domain of the remaining partial operation by
\[
\Omega_U^-=
(\widehat B_e\times\widehat B_e)
 \setminus(\Delta_U^-\cup E_U^-),
\qquad
R_U^-=U|_{\Omega_U^-}.
\]


\begin{proposition}\label{prop:r-max-regular}
Let $U\in\widetilde{\mathcal{U}}_{\mathrm{max}}^{r}$. Then $R_U^-$ is commutative and
increasing on its domain, and
\[
 \operatorname{Ran}(R_U^-)\subseteq\widehat B_e.
\]
It has the following compatibility properties.
\begin{enumerate}[label=\textup{(\roman*)}, itemsep=2pt]
 \item If $(x,y)\in\Delta_U^-$ and $(y,z)\in\Omega_U^-$, then
       $
       (x,R_U^-(y,z))\in\Delta_U^-.
       $
 \item If $(x,y)\in E_U^-$ and $(y,z)\in\Omega_U^-$, then
       \[
       (x,R_U^-(y,z))\in\Omega_U^-,
       \qquad R_U^-(x,R_U^-(y,z))=z.
       \]
 \item If $(x,y),(y,z)\in\Omega_U^-$, then, for
       $\Lambda\in\{\Delta_U^-,E_U^-\}$,
       \[
       (R_U^-(x,y),z)\in\Lambda
       \quad\Leftrightarrow\quad
       (x,R_U^-(y,z))\in\Lambda.
       \]
 If both
$
\bigl(R_U^-(x,y),z\bigr)
\ \text{and}\
\bigl(x,R_U^-(y,z)\bigr)
$
belong to $\Omega_U^-$, then
\[
R_U^-(R_U^-(x,y),z)
=
R_U^-(x,R_U^-(y,z)).
\]
 \item If $(p,q)\in E_U^-$, $(x,y)\in\Omega_U^-$, and
       $(p,q)\preceq(x,y)$, then
       $
       e\lhd R_U^-(x,y).
       $
\end{enumerate}
\end{proposition}

\begin{proof}
Commutativity and monotonicity are inherited from $U$. The range inclusion
follows from Proposition~\ref{prop:r-max-blocks} and the definitions of
$\Delta_U^-$, $E_U^-$, and $\Omega_U^-$.

Let $(x,y)\in\Delta_U^-$ and $(y,z)\in\Omega_U^-$. By
Theorem~\ref{thm:r-max-exceptional} and associativity,
\[
 U(x,R_U^-(y,z))=U(U(x,y),z)=U(g_U,z)=g_U.
\]
The localization in Proposition~\ref{prop:r-max-blocks} then gives (i).
If $(x,y)\in E_U^-$, the same calculation gives
\[
 U(x,R_U^-(y,z))=U(e,z)=z\in\widehat B_e,
\]
which proves (ii).

For (iii), associativity yields
\[
 U(R_U^-(x,y),z)=U(x,R_U^-(y,z)).                      \tag{\(**\)}
\]
For $(x,y)\in \widehat B_e\times\widehat B_e$, $U(x,y)$
 lies in exactly one of
$A_e$, $\{e\}$, and $\widehat B_e$; these three alternatives define
$\Delta_U^-$, $E_U^-$, and $\Omega_U^-$, respectively. The equality
in~$(**)$ therefore proves both equivalences and conditional associativity.

Finally, the assumptions in (iv) and monotonicity give
$e=U(p,q)\unlhd U(x,y)=R_U^-(x,y)$. Equality would put $(x,y)$ in $E_U^-$,
contrary to $(x,y)\in\Omega_U^-$. Thus the inequality is strict.
\end{proof}

\subsection{Structural properties of the class
\texorpdfstring{$\widetilde{\mathcal{U}}_{\mathrm{min}}^{r}$}{Umin-r}}

We next state the corresponding properties of
$\widetilde{\mathcal{U}}_{\mathrm{min}}^{r}$.
The notation introduced here will be used in the representation of this
class.

\begin{proposition}\label{prop:r-min-properties}
Let $U\in\widetilde{\mathcal{U}}_{\mathrm{min}}^{r}$ and put $J_e^-=I_e^1\cup I_e^3$.
Then the following statements hold.
\begin{enumerate}[label=\textup{(\roman*)}, itemsep=2pt]
 \item $U(b,x)\in B_e$ for $b\in B_e$ and $x\in X$,
       \[
       U(b,x)=b\qquad(b\in B_e,\ x\in\{e\}\cup\widehat A_e),
       \]
       and $U(1,x)=1$ for every $x\in X$.
 \item
       \[
       U(A_e\times A_e)\subseteq A_e,
       \qquad U(A_e\times I_e^3)\subseteq\widehat A_e.
       \]
       Hence a value on $\widehat A_e\times\widehat A_e$ can belong to
       $B_e\cup\{e\}$ only when both arguments belong to $I_e^3$.
 \item For $a\in\widehat A_e$ and $b\in B_e$,
       $a\lesssim b$ and $b\not\lesssim a$. If $b\in]e,1]$, then
       $a\lhd b$.
 \item
       \begin{align*}
       U(x,y)&\unlhd(x\vee e)\wedge(y\vee e)
          &&(x,y\in J_e^-),\\
       U(x,y)&\unlhd x
          &&(x\in J_e^-,\ y\in[0,e[),\\
       U(x,y)&\unlhd x\wedge y
          &&(x,y\in[0,e]).
       \end{align*}
 \item If $I_e^3\ne\varnothing$ and $s$ is given by
       Corollary~\ref{cor:r-boundary-elements}, then
       \[
       U(x,y)\unlhd s\qquad(x,y\in I_e^3).
       \]
\end{enumerate}
\end{proposition}


For $U\in\widetilde{\mathcal{U}}_{\mathrm{min}}^{r}$, define
\[
 \Delta_U^+=
 \{(x,y)\in I_e^3\times I_e^3\mid U(x,y)\in B_e\},
 \qquad
 E_U^+=
 \{(x,y)\in I_e^3\times I_e^3\mid U(x,y)=e\},
\]
and
\[
 \Omega_U^+=(\widehat A_e\times\widehat A_e)\setminus(\Delta_U^+\cup E_U^+),
 \qquad
 R_U^+=U|_{\Omega_U^+}.
\]

\begin{theorem}\label{thm:r-min-exceptional}
Let $U\in\widetilde{\mathcal{U}}_{\mathrm{min}}^{r}$. Then the following statements
hold.
\begin{enumerate}[label=\textup{(\roman*)}, itemsep=2pt]
 \item $\Delta_U^+$ and $E_U^+$ are symmetric upper sets in
       $(I_e^3\times I_e^3,\preceq)$ and $\Delta_U^+ \cap E_U^+=\varnothing$.
 \item If $\Delta_U^+\ne\varnothing$, then all its pairs have the same value
       $s_U\in]e,1]$. Moreover,
       \[
       U(b,s_U)=b\qquad(b\in B_e),
       \]
       $s_U$ is the smallest element of $B_e$, and
       \[
       s_U=e\vee c\qquad(c\in I_e^3).
       \]
 \item The symmetric set $E_U^+$ satisfies the unique-partner condition. Its
       support consists of maximal elements of $I_e^3$ and is an antichain.
 \item The set $\operatorname{supp}(\Delta_U^+)$ is an upper set in $I_e^3$. It is
       disjoint from, and elementwise incomparable with, the support of
       $E_U^+$.
 \item $R_U^+$ is commutative and increasing on its domain, and
       $\operatorname{Ran}(R_U^+)\subseteq\widehat A_e$. It satisfies the
       following conditions:
       \begin{enumerate}[label=\textup{(\alph*)}, itemsep=2pt]
        \item if $(x,y)\in\Delta_U^+$ and $(y,z)\in\Omega_U^+$, then
              $(x,R_U^+(y,z))\in\Delta_U^+$;
        \item if $(x,y)\in E_U^+$ and $(y,z)\in\Omega_U^+$, then
              \[
              (x,R_U^+(y,z))\in\Omega_U^+,
              \qquad R_U^+(x,R_U^+(y,z))=z;
              \]
        \item if $(x,y),(y,z)\in\Omega_U^+$, then, for
              $\Lambda\in\{\Delta_U^+,E_U^+\}$,
              \[
              (R_U^+(x,y),z)\in\Lambda
              \quad\Leftrightarrow\quad
              (x,R_U^+(y,z))\in\Lambda.
              \]
              If both
$
\bigl(R_U^+(x,y),z\bigr)
\  \text{and}\
\bigl(x,R_U^+(y,z)\bigr)
$
belong to $\Omega_U^+$, then
\[
R_U^+(R_U^+(x,y),z)
=
R_U^+(x,R_U^+(y,z)).
\]
        \item if $(x,y)\in\Omega_U^+$, $(p,q)\in E_U^+$, and
              $(x,y)\preceq(p,q)$, then
              $R_U^+(x,y)\lhd e$.
       \end{enumerate}
\end{enumerate}
\end{theorem}


\begin{corollary}\label{cor:r-disjoint-lattice}
The classes $\widetilde{\mathcal{U}}_{\mathrm{max}}^{r}$ and
$\widetilde{\mathcal{U}}_{\mathrm{min}}^{r}$ are disjoint. If $\unlhd$ is transitive, then
\[
 A_e=[0,e[,\qquad B_e=]e,1],\qquad I_e^3=I_e,
\]
and all the preceding statements reduce to their bounded-lattice
counterparts.
\end{corollary}

\begin{proof}
For $U\in\widetilde{\mathcal{U}}_{\mathrm{max}}^{r}$,
Proposition~\ref{prop:r-max-blocks} gives $U(0,1)=0$, whereas for
$U\in\widetilde{\mathcal{U}}_{\mathrm{min}}^{r}$,
Proposition~\ref{prop:r-min-properties} gives $U(0,1)=1$. Since $0\ne1$, no
uninorm belongs to both classes. If $\unlhd$ is transitive, then
$I_e^1=I_e^2=\varnothing$ and $I_e^3=I_e$. Hence, the stated
specialization follows immediately.

\end{proof}

\subsection{Decompositions for
\texorpdfstring{$\widetilde{\mathcal{U}}_{\mathrm{max}}^{r}$ and
$\widetilde{\mathcal{U}}_{\mathrm{min}}^{r}$}
{Umax-r and Umin-r}}

The preceding notation gives a decomposition of every operation in the
two classes. The following proposition lists the components uniquely
determined by a given uninorm; the converse is proved in
Section~\ref{sec:r-representation}.

\begin{proposition}\label{prop:r-canonical-decomposition}
Let $U\in\widetilde{\mathcal{U}}_{\mathrm{max}}^{r}$ and set
$
 T_U=U|_{(A_e\cup\{e\})\times(A_e\cup\{e\})}.
$
Then
\begin{equation}\label{eq:r-max-canonical}
U(x,y)=
\begin{cases}
T_U(x,y),
 & (x,y)\in(A_e\cup\{e\})\times(A_e\cup\{e\}),\\
x,
 & (x,y)\in\widehat B_e\times\{e\}
       \ \cup\ A_e\times\widehat B_e,\\
y,
 & (x,y)\in\{e\}\times\widehat B_e
       \ \cup\ \widehat B_e\times A_e,\\
g_U,
 & (x,y)\in\Delta_U^-,\\
e,
 & (x,y)\in E_U^-,\\
R_U^-(x,y),
 & (x,y)\in\Omega_U^-.
\end{cases}
\end{equation}
The fourth line is omitted when $\Delta_U^-=\varnothing$.  The set $A_e\cup\{e\}$ is closed under $T_U$, and $T_U$ is commutative,
associative, and increasing, with neutral element $e$. If
$\Delta_U^-\ne\varnothing$, then $T_U(a,g_U)=a$ for every $a\in A_e$.

If $U\in\widetilde{\mathcal{U}}_{\mathrm{min}}^{r}$ and
$S_U=U|_{(B_e\cup\{e\})\times(B_e\cup\{e\})}$, then
\begin{equation}\label{eq:r-min-canonical}
U(x,y)=
\begin{cases}
S_U(x,y),
 & (x,y)\in(B_e\cup\{e\})\times(B_e\cup\{e\}),\\
x,
 & (x,y)\in\widehat A_e\times\{e\}
       \ \cup\ B_e\times\widehat A_e,\\
y,
 & (x,y)\in\{e\}\times\widehat A_e
       \ \cup\ \widehat A_e\times B_e,\\
s_U,
 & (x,y)\in\Delta_U^+,\\
e,
 & (x,y)\in E_U^+,\\
R_U^+(x,y),
 & (x,y)\in\Omega_U^+.
\end{cases}
\end{equation}
Here the fourth line is omitted when $\Delta_U^+=\varnothing$.  The set $B_e\cup\{e\}$ is closed under $S_U$, and $S_U$ is commutative,
associative, and increasing, with neutral element $e$. If
$\Delta_U^+\ne\varnothing$, then $S_U(b,s_U)=b$ for every $b\in B_e$.
\end{proposition}

\begin{proof}
Closure follows from Proposition~\ref{prop:r-cone-closure}; commutativity,
associativity, monotonicity, and neutrality are inherited from $U$.
The mixed lines follow from neutrality,
commutativity, and Definition~\ref{def:umaxr-uminr-trellis}. The last three
lines in each formula partition the corresponding square according to whether
the value is $g_U$ or $s_U$, equals $e$, or is given by $R_U^-$ or $R_U^+$.
The identities
 $
T_U(a,g_U)=a$ for $a\in A_e
$
and
 $
S_U(b,s_U)=b$ for $b\in B_e$
follow from Theorems~\ref{thm:r-max-exceptional} and
\ref{thm:r-min-exceptional}, respectively.

\end{proof}

The preceding results determine the correct parameters of a representation.
For $\widetilde{\mathcal{U}}_{\mathrm{max}}^{r}$, these parameters are the lower
 component $U|_{(A_e\cup\{e\})\times(A_e\cup\{e\})}$, which takes values in
 $A_e\cup\{e\}$,
 the forced projection  on
$A_e\times(\{e\}\cup\widehat B_e)$, the sets
$\Delta_U^-$ and $E_U^-$ inside $I_e^3\times I_e^3$, and the partial operation
$R_U^-$ on $\Omega_U^-$. If $\Delta_U^-\ne\varnothing$, its value is not a
 free parameter:  it is the greatest element $g$ of $A_e$, characterized by  $g=e\wedge c$ for every $c\in I_e^3$.
Likewise, a
representation of $\widetilde{\mathcal{U}}_{\mathrm{min}}^{r}$ must use
 $U|_{(B_e\cup\{e\})\times(B_e\cup\{e\})}$, the  forced projection  on
$B_e\times(\{e\}\cup\widehat A_e)$, the sets
$\Delta_U^+$ and $E_U^+$ inside $I_e^3\times I_e^3$, and the partial operation
$R_U^+$ on $\Omega_U^+$; the common value on $\Delta_U^+$, when this set is
nonempty, is $s=e\vee c$ for every $c\in I_e^3$.

Thus, for
$U\in\widetilde{\mathcal{U}}_{\mathrm{max}}^{r}$
(resp.\
$U\in\widetilde{\mathcal{U}}_{\mathrm{min}}^{r}$), every pair
$(x,y)\in\widehat B_e\times\widehat B_e$
(resp.\
$(x,y)\in\widehat A_e\times\widehat A_e$)
such that $U(x,y)\notin\widehat B_e$
(resp.\ $U(x,y)\notin\widehat A_e$)
necessarily belongs to $I_e^3\times I_e^3$, rather than merely to
$I_e\times I_e$.
The elements of $I_e^1$ and $I_e^2$
already belong to $A_e$ and $B_e$, respectively. Any converse construction
must impose,
in addition to commutativity and monotonicity of the component operations,
the corresponding associativity conditions in
Proposition~\ref{prop:r-max-regular} and
Theorem~\ref{thm:r-min-exceptional}. They cover the associativity cases
involving more than one of the sets $\Delta_U^\pm$, $E_U^\pm$, and
$\Omega_U^\pm$.
Section~\ref{sec:r-representation} proves that the complete collection is
also sufficient.

\subsection{Representations for
\texorpdfstring{$\widetilde{\mathcal{U}}_{\mathrm{max}}^{r}$ and
$\widetilde{\mathcal{U}}_{\mathrm{min}}^{r}$}
{Umax-r and Umin-r}}

\label{sec:r-representation}


Throughout this section, let
$\mathbb{P}=(X,\unlhd,\wedge,\vee,0,1)$ be a bounded trellis and let
$e\in K\setminus\{0,1\}$.
We use the decomposition in
Proposition~\ref{prop:r-canonical-decomposition} to obtain piecewise
representations whose components are uniquely determined. For an operation
defined by cases, monotonicity must be verified both when two comparable
input pairs fall under the same case and when they fall under different
cases. The following definition records the latter comparisons. It does not
weaken monotonicity: provided that the map defining each individual case is
increasing on its domain, this condition is equivalent to monotonicity of
the resulting operation.

\begin{definition}\label{def:regional-order-compatibility}
Let $\mathscr P_0,\ldots,\mathscr P_m$ be a disjoint partition of $X\times X$ and
let $V_i:\mathscr P_i\to X$. The family
$\{(\mathscr P_i,V_i)\}_{i=0}^m$ is called \emph{regionally order-compatible}
if
\[
 \boldsymbol{x}\in\mathscr P_i,
 \quad \boldsymbol{y}\in\mathscr P_j,
 \quad \boldsymbol{x}\preceq\boldsymbol{y}
 \quad\Rightarrow\quad
 V_i(\boldsymbol{x})\unlhd V_j(\boldsymbol{y})
\]
whenever $i\ne j$. If every $V_i$ is increasing on $\mathscr P_i$,
regional order compatibility is equivalent to the map $V:X\times X\to X$ defined
by $V|_{\mathscr P_i}=V_i$ being increasing.
\end{definition}

\begin{lemma}\label{lem:r-local-cone-closure}
Let $e\in K\setminus\{0,1\}$. If
$
 H_-:(A_e\cup\{e\})\times(A_e\cup\{e\})\to A_e\cup\{e\}
$
is increasing and has neutral element $e$, then
$H_-(A_e\times A_e)\subseteq A_e$. Similarly, if
$
 H_+:(B_e\cup\{e\})\times(B_e\cup\{e\})\to B_e\cup\{e\}
$
is increasing and has neutral element $e$, then
$H_+(B_e\times B_e)\subseteq B_e$.
\end{lemma}

\begin{proof}
Let $x,y\in A_e$ and choose a directed path
$y=y_0\unlhd\cdots\unlhd y_n=e$. Its vertices belong to
$A_e\cup\{e\}$, and monotonicity along the path gives
$H_-(x,y)\lesssim H_-(x,e)=x$. If $H_-(x,y)=e$, then
$e\lesssim x\lesssim e$, producing a non-trivial cycle containing $e$.
The codomain of $H_-$ therefore gives $H_-(x,y)\in A_e$.

For $x,y\in B_e$, choose a directed path
$e=y_0\unlhd\cdots\unlhd y_n=y$. Its vertices belong to
$B_e\cup\{e\}$, and $x=H_+(x,e)\lesssim H_+(x,y)$. The equality
$H_+(x,y)=e$ would again put $e$ in a non-trivial cycle. Hence
$H_+(x,y)\in B_e$.
\end{proof}

\subsubsection{Representation of
\texorpdfstring{$\widetilde{\mathcal{U}}_{\mathrm{max}}^{r}$}{Umax-r}}

For symmetric sets $D,E\subseteq I_e^3\times I_e^3$, put
$
 \Omega^-_{D,E}=(\widehat B_e\times\widehat B_e)\setminus(D\cup E).
$
Let
\[
 T:(A_e\cup\{e\})\times(A_e\cup\{e\})\to A_e\cup\{e\},
 \qquad
 R:\Omega^-_{D,E}\to \widehat B_e
\]
be maps. We call $(T,D,E,R)$ an \emph{admissible
$\widetilde{\mathcal{U}}_{\mathrm{max}}^{r}$-datum} if the following conditions hold.
\begin{enumerate}[label=\textup{(Max\arabic*)}, itemsep=3pt]
 \item\label{cond:max-data-T}
 $T$ is commutative, associative, and increasing and has neutral element $e$.
 If $D\ne\varnothing$, then $A_e$ has a greatest element $g\in[0,e[$
 satisfying $T(a,g)=a$ for $a\in A_e$.

 \item\label{cond:max-data-layers}
 $D$ and $E$ are symmetric lower sets in $(I_e^3\times I_e^3,\preceq)$, $E$
 satisfies the unique-partner condition, and their supports are disjoint:
$
 \operatorname{supp}(D)\cap\operatorname{supp}(E)=\varnothing.
$

 \item\label{cond:max-data-core}
 Under Conditions~\ref{cond:max-data-T} and
 \ref{cond:max-data-layers}, define
 \begin{equation}\label{eq:r-upper-core-map}
  \Phi_{D,E,R}(x,y)=
  \begin{cases}
   g,      & (x,y)\in D,\\
   e,      & (x,y)\in E,\\
   R(x,y), & (x,y)\in\Omega^-_{D,E}.
  \end{cases}
 \end{equation}
 The first line is omitted when $D=\varnothing$. The map is well-defined by
 the disjoint-support condition, and $g$ exists when needed by
 Condition~\ref{cond:max-data-T}. The operation $R$ is commutative, and
 $\Phi_{D,E,R}$ is increasing on $\widehat B_e\times\widehat B_e$.

 \item\label{cond:max-data-assoc}
 For all $x,y,z\in \widehat B_e$, the following associativity conditions hold
 whenever the displayed pairs belong to the indicated domains:
 \begin{enumerate}[label=\textup{(\alph*)}, itemsep=2pt]
  \item If $(x,y)\in D$ and $(y,z)\in\Omega^-_{D,E}$, then
        $
        (x,R(y,z))\in D.
        $
  \item If $(x,y)\in E$ and $(y,z)\in\Omega^-_{D,E}$, then
        \[
        (x,R(y,z))\in\Omega^-_{D,E},
        \qquad R(x,R(y,z))=z.
        \]
  \item If $(x,y),(y,z)\in\Omega^-_{D,E}$, then, for
        $\Lambda\in\{D,E\}$,
        \[
        (R(x,y),z)\in\Lambda
        \quad\Leftrightarrow\quad
        (x,R(y,z))\in\Lambda.
        \]
        If $(R(x,y),z),(x,R(y,z))\in\Omega^-_{D,E}$, then
        \[
        R(R(x,y),z)=R(x,R(y,z)).
        \]
 \end{enumerate}

 \item\label{cond:max-data-cross-order}
 Define the four pairwise disjoint regions
 \begin{align*}
 \mathscr P_0^-&=(A_e\cup\{e\})\times(A_e\cup\{e\}),\\
 \mathscr P_1^-&=(\widehat B_e\times\{e\})\cup(A_e\times \widehat B_e),\\
 \mathscr P_2^-&=(\{e\}\times \widehat B_e)\cup(\widehat B_e\times A_e),\\
 \mathscr P_3^-&=\widehat B_e\times\widehat B_e,
 \end{align*}
 and the maps
 \[
 V_0^-=T,\qquad
 V_1^-(x,y)=x,\qquad
 V_2^-(x,y)=y,\qquad
 V_3^-=\Phi_{D,E,R}.
 \]
 The family $\{(\mathscr P_i^-,V_i^-)\}_{i=0}^3$ is regionally
 order-compatible.
\end{enumerate}



The admissibility conditions also imply
 $
T(A_e\times A_e)\subseteq A_e.
$
Moreover, if $D\ne\varnothing$, then
$
g=e\wedge c
\ \text{for every }c\in I_e^3.
$
Indeed, the first inclusion follows directly from
Lemma~\ref{lem:r-local-cone-closure}. To prove the second assertion,
suppose that $D\ne\varnothing$ and fix $c\in I_e^3$. Then
\[
(g,c)\preceq(e,c),\qquad
V_1^-(g,c)=g,\qquad
V_2^-(e,c)=c.
\]
Condition~\ref{cond:max-data-cross-order} therefore gives $g\unlhd c$.
If $z$ is any lower bound of $e$ and $c$, then $z=e$ is impossible because
$c\parallel e$; hence $z\in[0,e[\subseteq A_e$, and
\rev{since $g$ is the greatest element of $A_e$, we have $z\unlhd g$.}
Consequently,
\[
 g=e\wedge c\qquad(c\in I_e^3).
\]
Thus neither closure nor the meet identity is an independent admissibility
requirement.

The conditions on $D$ and $E$ retain all the support properties found in
Theorem~\ref{thm:r-max-exceptional}. Because $E$ is a lower set and partners
are unique, every element of $\operatorname{supp}(E)$ is minimal in
$I_e^3$; consequently, $\operatorname{supp}(E)$ is an antichain.
Moreover, $\operatorname{supp}(D)$ is a lower set. If
$p\in\operatorname{supp}(E)$ and $q\in\operatorname{supp}(D)$ were
comparable, then $p\unlhd q$ would imply
$p\in\operatorname{supp}(D)$, whereas $q\unlhd p$ would imply $q=p$ by
the minimality of $p$. Either conclusion contradicts
$\operatorname{supp}(D)\cap\operatorname{supp}(E)=\varnothing$; hence the
two supports are elementwise incomparable.{} Similarly, monotonicity of
$\Phi_{D,E,R}$ implies monotonicity of its restriction $R$.

 Condition~\ref{cond:max-data-cross-order} concerns only comparable
pairs belonging to different regions. The maps $V_0^-=T$ and
$V_3^-=\Phi_{D,E,R}$ are increasing by
Conditions~\ref{cond:max-data-T} and~\ref{cond:max-data-core}, whereas
$V_1^-$ and $V_2^-$ are coordinate projections.

The next theorem is an exact and uniquely recoverable piecewise
representation. Its forward implication identifies the datum forced by a
given uninorm, whereas its converse proves that every admissible datum
reconstructs a member of the class.

\begin{theorem}[Representation of
$\widetilde{\mathcal{U}}_{\mathrm{max}}^{r}$]
\label{thm:r-max-representation}
Let $\mathbb{P}=(X,\unlhd,\wedge,\vee,0,1)$ be a bounded trellis,
let $e\in K\setminus\{0,1\}$, and let $U:X\times X\to X$ be a binary operation.
Then
$U\in\widetilde{\mathcal{U}}_{\mathrm{max}}^{r}$ if and only if there exists an
admissible $\widetilde{\mathcal{U}}_{\mathrm{max}}^{r}$-datum $(T,D,E,R)$ such that
\begin{equation}\label{eq:r-max-representation}
U(x,y)=
\begin{cases}
T(x,y),
 & (x,y)\in(A_e\cup\{e\})\times(A_e\cup\{e\}),\\
x,
 & (x,y)\in(\widehat B_e\times\{e\})
       \ \cup\ (A_e\times \widehat B_e),\\
y,
 & (x,y)\in(\{e\}\times \widehat B_e)
       \ \cup\ (\widehat B_e\times A_e),\\
g,
 & (x,y)\in D,\\
e,
 & (x,y)\in E,\\
R(x,y),
 & (x,y)\in\Omega^-_{D,E}.
\end{cases}
\end{equation}
As above, the fourth line is omitted when $D=\varnothing$.
The representing datum is unique and is recovered from $U$ as
\[
 T=U|_{(A_e\cup\{e\})\times(A_e\cup\{e\})},\qquad
 D=\Delta_U^-,\qquad E=E_U^-,\qquad R=R_U^-.
\]
\end{theorem}

\begin{proof}
Suppose first that $U\in\widetilde{\mathcal{U}}_{\mathrm{max}}^{r}$. Take
\(T=U|_{(A_e\cup\{e\})\times(A_e\cup\{e\})}\),
\(D=\Delta_U^-\), \(E=E_U^-\), and \(R=R_U^-\).
Proposition~\ref{prop:r-cone-closure},
Theorem~\ref{thm:r-max-exceptional}, and
Proposition~\ref{prop:r-max-regular} give
the required properties in
Conditions~\ref{cond:max-data-T}--\ref{cond:max-data-assoc}; moreover,
$\Phi_{D,E,R}=U|_{\widehat B_e\times\widehat B_e}$ is increasing because $U$ is
increasing. In particular, if $D\ne\varnothing$, the common value of $U$ on
$D$ is the greatest element $g$ of $A_e$, where $g=e\wedge c$ for every $c\in I_e^3$,
and $T(a,g)=a$ for every $a\in A_e$.
The six lines of~\eqref{eq:r-max-representation} reproduce the decomposition
in Proposition~\ref{prop:r-canonical-decomposition}. Since $U$
is increasing, its restrictions to the four regions are regionally
order-compatible.
Thus Condition~\ref{cond:max-data-cross-order} also holds.

Conversely, suppose that an admissible datum is given. The four regions in
Condition~\ref{cond:max-data-cross-order} form a disjoint partition of $X\times X$;
inside the last region, $D,E,\Omega^-_{D,E}$ form a disjoint partition. Hence
formula~\eqref{eq:r-max-representation} defines $U$ unambiguously.

The symmetry of $D,E,\Omega^-_{D,E}$, together with the commutativity of
$T$ and $R$, shows that $U$ is commutative.
Condition~\ref{cond:max-data-T} and the two
mixed lines in~\eqref{eq:r-max-representation} show that $e$ is neutral.
Each $V_i^-$ is increasing on $\mathscr P_i^-$. Regional order
compatibility therefore implies, by
Definition~\ref{def:regional-order-compatibility}, that $U$ is increasing on
$X\times X$.

It remains to verify associativity. Neutrality deals with every triple
containing $e$, so let $x,y,z\in A_e\cup \widehat B_e$. If all three elements belong
to $A_e$, associativity follows from that of $T$. If $a,b\in A_e$ and
$h\in \widehat B_e$, then
\[
 U(U(a,b),h)=U(T(a,b),h)=T(a,b)
              =U(a,U(b,h)).
\]
Here $T(a,b)\in A_e$ follows from
Lemma~\ref{lem:r-local-cone-closure}. More explicitly, the three possible products
formed first from a pair of elements satisfy
\[
 U(T(a,b),h)=U(U(a,h),b)=U(U(b,h),a)=T(a,b),
\]
so every placement of $h$ is covered.

Next let $a\in A_e$ and $h,k\in \widehat B_e$. Since $U(a,h)=a$,
\(U(U(a,h),k)=a\).
\par\begin{samepage}
Put $v=U(h,k)$. If $(h,k)\in\Omega^-_{D,E}$, then
$v=R(h,k)\in \widehat B_e$ and $U(a,v)=a$. If $(h,k)\in E$, then $v=e$
and again $U(a,v)=a$. Finally, if $(h,k)\in D$, then $v=g$ and
$U(a,v)=T(a,g)=a$ by Condition~\ref{cond:max-data-T}. Thus associativity
also holds for triples having exactly one member in $A_e$:
$U(U(a,h),k)=U(U(a,k),h)=U(U(h,k),a)=a$.\par
\end{samepage}

It remains to consider $x,y,z\in \widehat B_e$. Classify each of $(x,y)$ and
$(y,z)$ as a member of $D$, $E$, or $\Omega^-_{D,E}$.

If both pairs belong to $D$, both sides of
\[
 U(U(x,y),z)=U(x,U(y,z))                                  \tag{\(*\)}
\]
equal $g$. If both belong to $E$, uniqueness of the $E$-partner of $y$ gives
$x=z$; the two sides of~$(*)$ are then $z$ and $x$, respectively, and hence
are equal. A pair in $D$ and a pair in $E$ cannot share the middle argument,
because the two supports are disjoint.

Assume that $(x,y)\in D$ and $(y,z)\in\Omega^-_{D,E}$. By
Condition~\ref{cond:max-data-assoc}\textup{(a)},
$(x,R(y,z))\in D$, so both sides of~$(*)$ equal $g$. The case
$(x,y)\in\Omega^-_{D,E}$ and $(y,z)\in D$ follows from the same condition after
reversing the variables and using symmetry and commutativity.

Assume that $(x,y)\in E$ and $(y,z)\in\Omega^-_{D,E}$. Then
Condition~\ref{cond:max-data-assoc}\textup{(b)} gives
\[
 U(x,U(y,z))=R(x,R(y,z))=z=U(e,z)=U(U(x,y),z).
\]
The case $(x,y)\in\Omega^-_{D,E}$ and $(y,z)\in E$ follows symmetrically.

Finally, suppose that $(x,y),(y,z)\in\Omega^-_{D,E}$. Put
\(p=(R(x,y),z)\) and \(q=(x,R(y,z))\).
 Condition~\ref{cond:max-data-assoc}\textup{(c)} gives
\[
 p\in D\Leftrightarrow q\in D,
 \qquad
 p\in E\Leftrightarrow q\in E.
\]
Thus, if either pair belongs to $D\cup E$, both belong to the same one of
$D$ and $E$, and the two sides of~$(*)$ are equal. If
$p,q\in\Omega^-_{D,E}$, then the same condition gives
$R(R(x,y),z)=R(x,R(y,z))$. This proves~$(*)$ in every case, and hence $U$ is
associative.

We have proved that $U$ is an increasing, associative, and commutative binary
operation with neutral element $e$. The second and third lines of
\eqref{eq:r-max-representation} give
\(U(a,h)=a\qquad(a\in A_e,\ h\in \widehat B_e)\),
so Definition~\ref{def:umaxr-uminr-trellis} yields
$U\in\widetilde{\mathcal{U}}_{\mathrm{max}}^{r}$. The ranges prescribed in the six
lines of~\eqref{eq:r-max-representation} identify $D$, $E$, and
$\Omega^-_{D,E}$ with $\Delta_U^-$, $E_U^-$, and $\Omega_U^-$,
respectively. The displayed restrictions then prove uniqueness.
\end{proof}

\begin{remark}\label{rem:r-max-representation-order}
Conditions~\ref{cond:max-data-assoc} and
\ref{cond:max-data-cross-order} are used, respectively, to verify
associativity and monotonicity. By
Conditions~\ref{cond:max-data-T} and~\ref{cond:max-data-core}, together with
the monotonicity of the coordinate projections, each $V_i^-$ is increasing
on $\mathscr P_i^-$. Regional order compatibility supplies precisely the
remaining comparisons required for $U$ to be increasing. For example, if
$x\in]0,e[$ and $y\in I_e^2\cup I_e^3$, comparison of $(y,x)$ with $(y,e)$
yields $x\unlhd y$. Since $x\in]0,e[$ and
$y\in I_e^2\cup I_e^3$, one has $x\ne y$, and hence $x\lhd y$. Thus it
also recovers the separation condition in
Proposition~\ref{prop:r-nonempty}. Example~\ref{ex:regional-compatibility}
shows that Conditions~\ref{cond:max-data-T}--\ref{cond:max-data-assoc}
alone do not imply these comparisons.

 Conditions~\ref{cond:max-data-assoc}\textup{(a)}--\textup{(c)} give the
remaining associativity cases on $\widehat B_e\times\widehat B_e\times\widehat B_e$. The $D$--$D$ cases
follow from the constant value $g$ on $D$, the $E$--$E$ cases from the
unique-partner property of $E$, and a pair in $D$ and a pair in $E$ cannot
share their middle argument because their supports are disjoint. Conditions
\textup{(a)}, \textup{(b)}, and \textup{(c)} handle, respectively, the
$D$--$\Omega$, $E$--$\Omega$, and $\Omega$--$\Omega$ cases, together with
their symmetric versions. They determine whether the nested pairs belong to
$D$, $E$, or $\Omega^-_{D,E}$ and, when both belong to
$\Omega^-_{D,E}$, require equality of the two $R$-values. This verifies
associativity on $\widehat B_e\times\widehat B_e\times\widehat B_e$; it does not assert that $R$ is an
everywhere-defined associative operation. For finite trellises,
Conditions~\ref{cond:max-data-assoc} and
\ref{cond:max-data-cross-order} require only finitely many checks, although
this need not be shorter than verifying associativity and monotonicity of
$U$ directly.
\end{remark}

\begin{corollary}[Construction with $D=E=\varnothing$]
\label{cor:r-max-regular-core-construction}
Let $T:(A_e\cup\{e\})\times(A_e\cup\{e\})\to A_e\cup\{e\}$ satisfy
Condition~\ref{cond:max-data-T} with $D=\varnothing$, and let
$R:\widehat B_e\times\widehat B_e\to\widehat B_e$ be commutative, associative, and
increasing. If the four maps specified in
Condition~\ref{cond:max-data-cross-order}, with
$\Phi_{\varnothing,\varnothing,R}=R$, are regionally order-compatible, then
\[
 U(x,y)=
 \begin{cases}
  T(x,y), &(x,y)\in(A_e\cup\{e\})\times(A_e\cup\{e\}),\\
  x, &(x,y)\in(\widehat B_e\times\{e\})
        \cup(A_e\times\widehat B_e),\\
  y, &(x,y)\in(\{e\}\times\widehat B_e)
        \cup(\widehat B_e\times A_e),\\
  R(x,y), &(x,y)\in\widehat B_e\times\widehat B_e
 \end{cases}
\]
belongs to $\widetilde{\mathcal{U}}_{\mathrm{max}}^{r}$. Conversely, every member of
$\widetilde{\mathcal{U}}_{\mathrm{max}}^{r}$ with
$\Delta_U^-=E_U^-=\varnothing$ has this form, uniquely, with
\[
 T=U|_{(A_e\cup\{e\})\times(A_e\cup\{e\})},\qquad
 R=U|_{\widehat B_e\times\widehat B_e}.
\]
In particular, if $I_e^3=\varnothing$, this
construction represents the whole class $\widetilde{\mathcal{U}}_{\mathrm{max}}^{r}$.
\end{corollary}

\begin{proof}
Set $D=E=\varnothing$ in
Theorem~\ref{thm:r-max-representation}. Conditions
\ref{cond:max-data-layers} and the first two clauses of
Condition~\ref{cond:max-data-assoc} are then void. In its third clause, the
membership equivalences involving $D$ and $E$ are tautological, and the only
non-vacuous requirement is the associativity of $R$.
\end{proof}

The following example explains why compatibility between
different regions is required in
Corollary~\ref{cor:r-max-regular-core-construction}.

\begin{example}
\label{ex:regional-compatibility}
\rev{Let $Y=\{0,a,e,b,c,1\}$, and let
$\Delta_Y=\{(y,y):y\in Y\}\subseteq Y\times Y$ denote the diagonal
relation on $Y$.}
Define a pseudo-order $\unlhd$ by
\[
 \unlhd\;=\Delta_Y
 \cup\bigl(\{0\}\times Y\bigr)
 \cup\bigl(Y\times\{1\}\bigr)
 \cup\{(a,e),(a,b),(e,b),(b,c)\}.
\]
The displayed relation is identical, under the given labels, to the one in
Example~\ref{ex:finite-utop-star}; its Hasse-type diagram is therefore the
diagram already shown in Figure~\ref{fig:hasse-utop-star}.

The only incomparable unordered pairs are $\{a,c\}$ and $\{e,c\}$; both
have meet $0$ and join $1$. Hence this relation defines a bounded trellis.
It is proper because
\[
 (e\wedge b)\wedge c=0\ne e=e\wedge(b\wedge c).
\]
The principal sections of $e$ are
$\{x\in Y:x\unlhd e\}=\{0,a,e\}$ and
$\{x\in Y:e\unlhd x\}=\{e,b,1\}$; the relation
$a\unlhd b$ makes $e$
middle-transitive. Since the strict relation is acyclic, $e$ belongs to $K$.
The corresponding regions are
\(A_e=\{0,a\}\), \(B_e=\{b,c,1\}\)  and \(I_e^3=\varnothing\).

Let $T$ be the minimum operation on the auxiliary chain $0<a<e$, let $Q$
be the maximum operation on the auxiliary chain $b<c<1$, and let $U$ be the
operation defined by four cases in
Corollary~\ref{cor:r-max-regular-core-construction},
but without imposing regional order compatibility. The specified quadruple satisfies
Conditions~\ref{cond:max-data-T}--\ref{cond:max-data-assoc} with
$D=E=\varnothing$ and $R=Q$. Writing $\preceq$ for the
induced product pseudo-order, we nevertheless have
\((a,b)\preceq(e,c)\), \(U(a,b)=a\), and \(U(e,c)=c\),
where $a\ntrianglelefteq c$. Thus $U$ is not increasing.
This failure is precisely the one ruled out by
Condition~\ref{cond:max-data-cross-order}. It also agrees with
Proposition~\ref{prop:r-nonempty}, since the required direct comparison
$a\unlhd c$ fails even though $a$ reaches $c$.
\end{example}

\subsubsection{Representation of
\texorpdfstring{$\widetilde{\mathcal{U}}_{\mathrm{min}}^{r}$}{Umin-r}}

\rev{The representation below is the explicit dual of
Theorem~\ref{thm:r-max-representation}.}
For sets $D,E\subseteq I_e^3\times I_e^3$, put
\[
 \Omega^+_{D,E}=(\widehat A_e\times\widehat A_e)\setminus(D\cup E).
\]
Let
\[
 S:(B_e\cup\{e\})\times(B_e\cup\{e\})\to B_e\cup\{e\},
 \qquad
 F:\Omega^+_{D,E}\to\widehat A_e
\]
be maps.
The letter $F$ denotes the partial operation in the representing datum; when
the datum is recovered from a given uninorm $U$, one has $F=R_U^+$.
\begin{definition}\label{def:r-min-admissible}
An \emph{admissible $\widetilde{\mathcal{U}}_{\mathrm{min}}^{r}$-datum} is a quadruple
$(S,D,E,F)$ satisfying the following conditions.
\begin{enumerate}[label=\textup{(Min\arabic*)}, itemsep=3pt]
 \item\label{cond:min-data-S}
 $S$ is commutative, associative, and increasing and has neutral element $e$.
 If $D\ne\varnothing$, then $B_e$ has a smallest element $s\in]e,1]$
  satisfying
   $S(b,s)=b$ for $b\in B_e$.

 \item\label{cond:min-data-layers}
 $D$ and $E$ are symmetric upper sets in
 $(I_e^3\times I_e^3,\preceq)$, $E$ satisfies the unique-partner condition, and
 their supports are disjoint:
$
  \operatorname{supp}(D)\cap\operatorname{supp}(E)=\varnothing.
$

 \item\label{cond:min-data-core}
 Under Conditions~\ref{cond:min-data-S} and
 \ref{cond:min-data-layers}, define
 \begin{equation}\label{eq:r-lower-core-map}
  \Psi_{D,E,F}(x,y)=
  \begin{cases}
   s,      & (x,y)\in D,\\
   e,      & (x,y)\in E,\\
   F(x,y), & (x,y)\in\Omega^+_{D,E}.
  \end{cases}
 \end{equation}
 The first line is omitted when $D=\varnothing$. Since the
 disjoint-support condition implies $D\cap E=\varnothing$, the sets $D$,
 $E$, and $\Omega^+_{D,E}$ form a disjoint partition of
 $\widehat A_e\times\widehat A_e$; hence $\Psi_{D,E,F}$ is well defined. The element $s$
 exists when needed by Condition~\ref{cond:min-data-S}. The partial operation
 $F$ is commutative, and $\Psi_{D,E,F}$ is increasing on
 $\widehat A_e\times\widehat A_e$.

 \item\label{cond:min-data-assoc}
 For all $x,y,z\in\widehat A_e$, the following associativity conditions hold
 whenever the pairs in their hypotheses belong to the
 indicated domains:
 \begin{enumerate}[label=\textup{(\alph*)}, itemsep=2pt]
  \item If $(x,y)\in D$ and $(y,z)\in\Omega^+_{D,E}$, then
        $
        (x,F(y,z))\in D.
        $
  \item If $(x,y)\in E$ and $(y,z)\in\Omega^+_{D,E}$, then
        \[
        (x,F(y,z))\in\Omega^+_{D,E},
        \qquad F(x,F(y,z))=z.
        \]
  \item If $(x,y),(y,z)\in\Omega^+_{D,E}$, then, for
        $\Lambda\in\{D,E\}$,
        \[
        (F(x,y),z)\in\Lambda
        \quad\Leftrightarrow\quad
        (x,F(y,z))\in\Lambda.
        \]
        If $(F(x,y),z),(x,F(y,z))\in\Omega^+_{D,E}$, then
        $
        F(F(x,y),z)=F(x,F(y,z)).
        $
 \end{enumerate}

 \item\label{cond:min-data-cross-order}
 Define the four pairwise disjoint regions
 \begin{align*}
  \mathscr P_0^+&=(B_e\cup\{e\})\times(B_e\cup\{e\}),\\
  \mathscr P_1^+&=(\widehat A_e\times\{e\})
                    \cup(B_e\times\widehat A_e),\\
  \mathscr P_2^+&=(\{e\}\times\widehat A_e)
                    \cup(\widehat A_e\times B_e),\\
  \mathscr P_3^+&=\widehat A_e\times\widehat A_e,
 \end{align*}
 and the maps
 \[
  V_0^+=S,\qquad V_1^+(x,y)=x,\qquad
  V_2^+(x,y)=y,\qquad V_3^+=\Psi_{D,E,F}.
 \]
 The family $\{(\mathscr P_i^+,V_i^+)\}_{i=0}^3$ is regionally
 order-compatible.
\end{enumerate}
\end{definition}

The admissibility conditions also imply
$S(B_e\times B_e)\subseteq B_e$.
Moreover, if $D\ne\varnothing$, then
$
g=e\vee c
\ \text{for every }c\in I_e^3.
$
Indeed, the first inclusion follows directly from
Lemma~\ref{lem:r-local-cone-closure}. To prove the second assertion,
suppose that $D\ne\varnothing$ and fix $c\in I_e^3$. Then
%
%
%
\[
 (c,e)\preceq(c,s),\qquad V_1^+(c,e)=c,
 \qquad V_2^+(c,s)=s.
\]
Regional order compatibility gives $c\unlhd s$. If $z$ is any upper bound of
$e$ and $c$, then $z=e$ is impossible because $c\parallel e$; hence
$z\in]e,1]\subseteq B_e$, and
\rev{since $s$ is the smallest element of $B_e$, we have $s\unlhd z$.}
Therefore,
\[
 s=e\vee c\qquad(c\in I_e^3).
\]
Thus neither closure nor the join identity is an independent admissibility
requirement.

The conditions on $D$ and $E$ also recover the support properties in
Theorem~\ref{thm:r-min-exceptional}. Because $E$ is an upper set and partners
are unique, every element of $\operatorname{supp}(E)$ is maximal in $I_e^3$;
hence $\operatorname{supp}(E)$ is an antichain. Moreover,
$\operatorname{supp}(D)$ is an upper set. If
$p\in\operatorname{supp}(E)$ and $q\in\operatorname{supp}(D)$ were
comparable, then $q\unlhd p$ would imply
$p\in\operatorname{supp}(D)$, whereas $p\unlhd q$ would imply $q=p$ by
the maximality of $p$. Either conclusion contradicts
$\operatorname{supp}(D)\cap\operatorname{supp}(E)=\varnothing$; hence the
two supports are elementwise incomparable.{} Monotonicity of
$\Psi_{D,E,F}$ implies monotonicity of its
restriction $F$.

Condition~\ref{cond:min-data-assoc} provides the remaining cases needed to
verify associativity on
$\widehat A_e\times\widehat A_e\times\widehat A_e$, whereas
Condition~\ref{cond:min-data-cross-order} supplies the comparisons between
different formula regions required for the monotonicity of the assembled
operation.  
Regional order compatibility supplies the comparisons required for $U$ to
be increasing when the two pairs belong to different regions. The maps
$V_0^+=S$ and $V_3^+=\Psi_{D,E,F}$ are increasing by
Conditions~\ref{cond:min-data-S} and~\ref{cond:min-data-core}, whereas
$V_1^+$ and $V_2^+$ are coordinate projections. Together, these conditions
ensure associativity and monotonicity.{} In particular, regional order compatibility
forces the feasibility condition in Proposition~\ref{prop:r-nonempty}:
compare $(x,e)$ with $(x,y)$ for $x\in I_e^1\cup I_e^3$ and $y\in]e,1[$.

The same interpretation applies to the following representation: the datum
is uniquely recovered from the uninorm, and every admissible datum yields a
member of the stated class.

\begin{theorem}[Representation of
$\widetilde{\mathcal{U}}_{\mathrm{min}}^{r}$]
\label{thm:r-min-representation}
Let $\mathbb{P}=(X,\unlhd,\wedge,\vee,0,1)$ be a bounded trellis,
let $e\in K\setminus\{0,1\}$, and let $U:X\times X\to X$ be a binary operation.
Then
$U\in\widetilde{\mathcal{U}}_{\mathrm{min}}^{r}$ if and only if there exists an
admissible $\widetilde{\mathcal{U}}_{\mathrm{min}}^{r}$-datum $(S,D,E,F)$ such that
\begin{equation}\label{eq:r-min-representation}
U(x,y)=
\begin{cases}
S(x,y),
 & (x,y)\in(B_e\cup\{e\})\times(B_e\cup\{e\}),\\
x,
 & (x,y)\in(\widehat A_e\times\{e\})
       \ \cup\ (B_e\times \widehat A_e),\\
y,
 & (x,y)\in(\{e\}\times \widehat A_e)
       \ \cup\ (\widehat A_e\times B_e),\\
s,
 & (x,y)\in D,\\
e,
 & (x,y)\in E,\\
F(x,y),
 & (x,y)\in\Omega^+_{D,E}.
\end{cases}
\end{equation}
The fourth line is omitted when $D=\varnothing$.
The representing datum is unique and is recovered from $U$ as
\[
 S=U|_{(B_e\cup\{e\})\times(B_e\cup\{e\})},\qquad
 D=\Delta_U^+,\qquad E=E_U^+,\qquad F=R_U^+.
\]
\end{theorem}


\begin{corollary}[Construction for
$\widetilde{\mathcal{U}}_{\mathrm{min}}^{r}$ with $D=E=\varnothing$]
\label{cor:r-min-regular-core-construction}
Let $S:(B_e\cup\{e\})\times(B_e\cup\{e\})\to B_e\cup\{e\}$ be commutative,
associative, and increasing and have neutral element $e$. Let
$F:\widehat A_e\times\widehat A_e\to\widehat A_e$ be commutative, associative, and
increasing. If the four maps specified in
Condition~\ref{cond:min-data-cross-order}, with
$\Psi_{\varnothing,\varnothing,F}=F$, are regionally order-compatible, then
\[
 U(x,y)=
 \begin{cases}
  S(x,y), &(x,y)\in(B_e\cup\{e\})\times(B_e\cup\{e\}),\\
  x, &(x,y)\in(\widehat A_e\times\{e\})
        \cup(B_e\times\widehat A_e),\\
  y, &(x,y)\in(\{e\}\times\widehat A_e)
        \cup(\widehat A_e\times B_e),\\
  F(x,y), &(x,y)\in\widehat A_e\times\widehat A_e
 \end{cases}
\]
belongs to $\widetilde{\mathcal{U}}_{\mathrm{min}}^{r}$. Conversely, every member of
$\widetilde{\mathcal{U}}_{\mathrm{min}}^{r}$ with
$\Delta_U^+=E_U^+=\varnothing$ has this form uniquely, with
\[
 S=U|_{(B_e\cup\{e\})\times(B_e\cup\{e\})},\qquad
 F=U|_{\widehat A_e\times\widehat A_e}.
\]
In particular, if $I_e^3=\varnothing$, this construction represents the
whole class $\widetilde{\mathcal{U}}_{\mathrm{min}}^{r}$.
\end{corollary}


The following example illustrates the possible values on
$I_e^3\times I_e^3$ in the representations of
Theorems~\ref{thm:r-max-representation} and~\ref{thm:r-min-representation}
for proper bounded trellises. In particular, it exhibits nonempty sets
$\Delta_U^-$, $\Delta_U^+$, and $E_U^-$.

\begin{example}
\label{ex:proper-trellis-r-classes}
Let
\(X=\{0,a,b,e,c,1\}\).
Define $\unlhd$ by
\begin{equation}\label{eq:example-pseudo-order}
\unlhd\;=\Delta_X
 \cup\bigl(\{0\}\times X\bigr)
 \cup\bigl(X\times\{1\}\bigr)
 \cup\{(a,b),(b,e),(a,c),(b,c)\},
\end{equation}
\rev{where $\Delta_X=\{(x,x):x\in X\}\subseteq X\times X$ is the
diagonal relation.} The relation is reflexive,
antisymmetric, and bounded, but it is not transitive: $a\unlhd b$ and
$b\unlhd e$, whereas $a\ntrianglelefteq e$.  
Figure~\ref{fig:hasse-r-base} displays this failure of transitivity by the
dashed curve joining $a$ and $e$.

\begin{figure}[!htbp]
\centering
\includegraphics[width=0.31\textwidth]{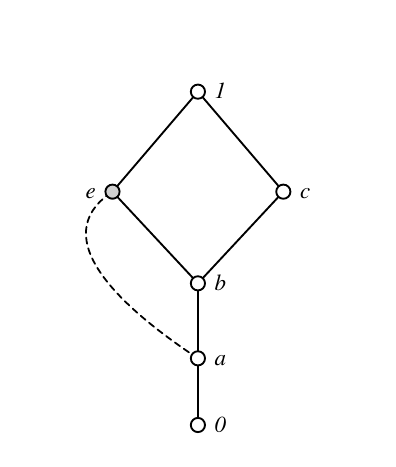}
\caption{Hasse-type diagram of the bounded trellis $\mathbb P$ in
Example~\ref{ex:proper-trellis-r-classes}.}
\label{fig:hasse-r-base}
\end{figure}

Every pair has an infimum and a supremum. For a comparable pair these
are its smaller and larger elements, respectively. The only
incomparable unordered pairs are $\{a,e\}$ and $\{c,e\}$, and
\[
 a\wedge e=0,\quad a\vee e=1,
 \qquad
 c\wedge e=b,\quad c\vee e=1.
\]
Hence $\mathbb{P}=(X,\unlhd,\wedge,\vee,0,1)$ is a bounded trellis.
It is a proper trellis, since
\[
 (a\wedge b)\wedge e=0
 \ne a=a\wedge(b\wedge e).
\]
The lower and upper sections of $e$ are, respectively,
$[0,e]=\{0,b,e\}$ and $[e,1]=\{e,1\}$. Therefore $e$ is middle-transitive. The
strict relation in \eqref{eq:example-pseudo-order} is acyclic, so
$e\in K$. Moreover,
\(I_e^1=\{a\},\ I_e^2=\varnothing, \ I_e^3=\{c\}\),
and consequently
\(A_e=\{0,a,b\},\ B_e=\{1\}\).

Define two binary operations
$U_{\mathrm{max}},U_{\mathrm{min}}:X\times X\to X$
by the following tables. The subscripts indicate the principal classes
$\widetilde{\mathcal{U}}_{\mathrm{max}}^{r}$ and
$\widetilde{\mathcal{U}}_{\mathrm{min}}^{r}$ illustrated by these
operations, respectively; they do not denote the usual maximum and
minimum operations. The rows and columns are ordered as
$0,a,b,e,c,1$.
\[
\begin{array}{c|cccccc}
U_{\mathrm{max}}&0&a&b&e&c&1\\ \hline
0&0&0&0&0&0&0\\
a&0&a&a&a&a&a\\
b&0&a&b&b&b&b\\
e&0&a&b&e&c&1\\
c&0&a&b&c&b&c\\
1&0&a&b&1&c&1
\end{array}
\qquad
\begin{array}{c|cccccc}
U_{\mathrm{min}}&0&a&b&e&c&1\\ \hline
0&0&0&0&0&c&1\\
a&0&a&a&a&c&1\\
b&0&a&b&b&c&1\\
e&0&a&b&e&c&1\\
c&c&c&c&c&1&1\\
1&1&1&1&1&1&1
\end{array}.
\]
Both tables are symmetric and have neutral element $e$. Because the
operations are commutative, it suffices, for each unordered triple, to compare
the values obtained by forming each of its three pairs first. For
$U_{\mathrm{max}}$, the restriction to
$\{0,a,b,e\}$ is the minimum operation for the auxiliary chain
$0<a<b<e$; moreover,
$U_{\mathrm{max}}(h,t)=t$ for
$h\in\{c,1\}$ and $t\in\{0,a,b\}$; and
\(U_{\mathrm{max}}(c,c)=b\), \(U_{\mathrm{max}}(c,1)=c\), and
\(U_{\mathrm{max}}(1,1)=1\).
After neutral occurrences of $e$ are deleted, every evaluation of a
triple containing an element of $\{0,a,b\}$ gives the least such element. If
no such element occurs, the remaining triples use only $c$ and $1$, and these three identities give the same value for every choice of the first
pair. For $U_{\mathrm{min}}$, $1$ is absorbing, the restriction to
$\{0,a,b,e\}$ is again the auxiliary minimum, $c$ absorbs
$\{0,a,b\}$, and $U_{\mathrm{min}}(c,c)=1$. Thus any triple containing $1$ or at least two occurrences of $c$ evaluates
to $1$. \par\begin{samepage}
If it contains exactly one $c$ and no $1$, its value is $c$. Every
remaining triple is evaluated by the auxiliary minimum. Thus both operations
are associative.\par
\end{samepage}

For a binary operation $W$ on $X$, put
\[
 \mathcal Q_W=\{(W(x_1,x_2),W(y_1,y_2))\mid
 (x_1,y_1),(x_2,y_2)\in\unlhd\}.
\]
The relation $\unlhd$ has $19$ ordered pairs: besides the reflexive pairs
and the pairs involving $0$ or $1$, its strict comparisons are
\[
 a\lhd b,\qquad b\lhd e,\qquad a\lhd c,\qquad b\lhd c.
\]
Reading the two displayed tables gives
\[
 \mathcal Q_{U_{\mathrm{max}}}=\mathcal Q_{U_{\mathrm{min}}}
 =\Delta_X\cup(\{0\}\times X)\cup(X\times\{1\})
  \cup\{(a,b),(b,e),(a,c),(b,c)\}=\unlhd.
\]
Consequently, $(x_1,x_2)\preceq(y_1,y_2)$ implies
$W(x_1,x_2)\unlhd W(y_1,y_2)$ for
$W\in\{U_{\mathrm{max}},U_{\mathrm{min}}\}$; hence both operations are
jointly increasing. Therefore both operations are uninorms.

Since $\widehat B_e=\{c,1\}$ and
\(U_{\mathrm{max}}(x,y)=x\) for $x\in A_e$ and $y\in\widehat B_e$,
we have $U_{\mathrm{max}}\in\widetilde{\mathcal{U}}_{\mathrm{max}}^{r}$. Its lower component
$T=U_{\mathrm{max}}|_{(A_e\cup\{e\})\times(A_e\cup\{e\})}$ is the minimum operation on the auxiliary
chain $0<a<b<e$. Its sets $\Delta_{U_{\mathrm{max}}}^-$ and
$E_{U_{\mathrm{max}}}^-$ and the remaining representation data are
\[
 \Delta_{U_{\mathrm{max}}}^-=\{(c,c)\},
 \qquad E_{U_{\mathrm{max}}}^-=\varnothing,
 \qquad g_{U_{\mathrm{max}}}=e\wedge c=b,
\]
and
\[
\begin{aligned}
 \Omega_{U_{\mathrm{max}}}^-&=\{(c,1),(1,c),(1,1)\},\\
 R_{U_{\mathrm{max}}}^-(c,1)&=R_{U_{\mathrm{max}}}^-(1,c)=c,\\
 R_{U_{\mathrm{max}}}^-(1,1)&=1.
\end{aligned}
\]
Condition~\ref{cond:max-data-assoc}\textup{(a)} is non-vacuous in this
example. Indeed, $(c,c)\in\Delta_{U_{\mathrm{max}}}^-$ and
$(c,1)\in\Omega_{U_{\mathrm{max}}}^-$; since
$R_{U_{\mathrm{max}}}^-(c,1)=c$, one obtains
$
 \bigl(c,R_{U_{\mathrm{max}}}^-(c,1)\bigr)=(c,c)
 \in\Delta_{U_{\mathrm{max}}}^-,
$
which is precisely Condition~\ref{cond:max-data-assoc}\textup{(a)} for
$x=y=c$ and $z=1$.
Thus the incomparable reachable element $a\in I_e^1$ is handled by the
lower component, whereas only $(c,c)$ belongs to
$\Delta_{U_{\mathrm{max}}}^-$.

Similarly, $\widehat A_e=\{0,a,b,c\}$ and
\(U_{\mathrm{min}}(1,x)=1\) for  $x\in\widehat A_e$,
so $U_{\mathrm{min}}\in\widetilde{\mathcal{U}}_{\mathrm{min}}^{r}$. In this case the upper
component $S=U_{\mathrm{min}}|_{(B_e\cup\{e\})\times(B_e\cup\{e\})}$ is the maximum operation on the
chain $e<1$, and
\[
 \Delta_{U_{\mathrm{min}}}^+=\{(c,c)\},
 \qquad E_{U_{\mathrm{min}}}^+=\varnothing,
 \qquad s_{U_{\mathrm{min}}}=e\vee c=1.
\]
Moreover, $\Omega_{U_{\mathrm{min}}}^+=(\widehat A_e\times\widehat A_e)\setminus\{(c,c)\}$
and $R_{U_{\mathrm{min}}}^+=U_{\mathrm{min}}|_{\Omega_{U_{\mathrm{min}}}^+}$.
Writing $F=R_{U_{\mathrm{min}}}^+$, Condition
\ref{cond:min-data-assoc}\textup{(a)} is also non-vacuous: taking
$x=y=c$ and $z=0$, one has $(c,c)\in\Delta_{U_{\mathrm{min}}}^+$,
$(c,0)\in\Omega_{U_{\mathrm{min}}}^+$, and $F(c,0)=c$, whence
\[
 (c,F(c,0))=(c,c)\in\Delta_{U_{\mathrm{min}}}^+.
\]
By Theorems~\ref{thm:r-max-representation} and
\ref{thm:r-min-representation}, the corresponding data
\[
 \bigl(T,\Delta^-_{U_{\mathrm{max}}},E^-_{U_{\mathrm{max}}},
       R^-_{U_{\mathrm{max}}}\bigr)
 \quad\text{and}\quad
 \bigl(S,\Delta^+_{U_{\mathrm{min}}},E^+_{U_{\mathrm{min}}},F\bigr)
\]
are the unique representing data recovered from the two operation tables;
they are not merely convenient choices.
\par\begin{samepage}
 This example therefore satisfies
$\Delta_{U_{\mathrm{max}}}^-\ne\varnothing$ and
$\Delta_{U_{\mathrm{min}}}^+\ne\varnothing$ on the proper trellis
$\mathbb{P}$. It also illustrates
why $I_e^1$ and $I_e^3$ must be kept distinct in the representation.\par
\end{samepage}

An example with $E_U^-\ne\varnothing$ is obtained by a slight modification. Let
\[
 \unlhd'\;=\unlhd\setminus\{(a,c)\}
\]
and let $V$ be obtained from the displayed table for $U_{\mathrm{max}}$ by
changing only
\[
V(c,c)=e,\qquad V(c,1)=V(1,c)=1.
\]
The relation $\unlhd'$ again defines a proper bounded trellis, denoted by
$\mathbb{P}'$. Its
incomparable unordered pairs are $\{a,e\}$, $\{a,c\}$, and $\{c,e\}$; their
meets and joins are
\[
 a\wedge' e=a\wedge' c=0,\qquad
 a\vee' e=a\vee' c=1,\qquad
 c\wedge' e=b,\qquad c\vee' e=1.
\]
Figure~\ref{fig:hasse-r-variant} shows that deleting $(a,c)$ introduces
the additional dashed curve joining $a$ and $c$.

\begin{figure}[!htbp]
\centering
\includegraphics[width=0.31\textwidth]{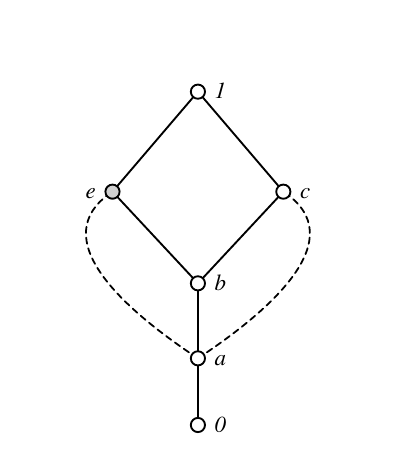}
\caption{Hasse-type diagram of the modified bounded trellis $\mathbb P'$ in
Example~\ref{ex:proper-trellis-r-classes}.}
\label{fig:hasse-r-variant}
\end{figure}

The decomposition is unchanged, and $e\in K$. The operation $V$ is
commutative and has neutral element $e$. On $\{0,a,b,e\}$ it is the minimum
for the auxiliary chain $0<a<b<e$, and every element of $\{0,a,b\}$ absorbs
both $c$ and $1$. On $\{e,c,1\}$, the identities
\[
 V(c,c)=e,\qquad V(c,1)=1,\qquad V(1,1)=1
\]
show directly that the restriction is associative. These observations cover
every unordered triple, so $V$ is associative.
With $\unlhd'=\unlhd\setminus\{(a,c)\}$, we obtain
\[
\begin{aligned}
\mathcal Q_V
&=\{(V(x_1,x_2),V(y_1,y_2))\mid
  (x_1,y_1),(x_2,y_2)\in\unlhd'\}\\
&=\Delta_X\cup(\{0\}\times X)\cup(X\times\{1\})
  \cup\{(a,b),(b,e),(b,c)\}\\
&=\unlhd'.
\end{aligned}
\]

Thus $V$ is jointly increasing. Hence, on $\mathbb{P}'$,
$V\in\widetilde{\mathcal{U}}_{\mathrm{max}}^{r}$. Here
$\Delta_V^-=\varnothing$, whereas
\(E_V^-=\{(c,c)\},\ \Omega_V^-=\{(c,1),(1,c),(1,1)\}\)
and $R_V^-$ is constantly $1$ on $\Omega_V^-$. Thus $E_V^-$ satisfies the
unique-partner condition, with $(c,c)$ as its only pair. Taking the order
dual of this construction gives a proper bounded trellis on which the same
operation $V$ belongs to $\widetilde{\mathcal{U}}_{\mathrm{min}}^{r}$.
Indeed, duality interchanges $A_e$ and $B_e$ while leaving $I_e^3$
unchanged, and hence $E_V^+=\{(c,c)\}\ne\varnothing$ in the dual trellis.
\end{example}

\section{Comparison and lattice specialization}\label{sec:comparison}

In this section, we first compare the representation structures of the
four principal classes and determine which intersections between the two
pairs are empty and which may be nonempty. We then clarify the precise
relationship between the classes considered here and the corresponding
classes of uninorms on bounded psosets introduced in~\cite{KL26}.
Finally, we show how the trellis classes and their representation results
specialize to bounded lattices.

%
%

\rev{We begin with a comparison of the four representations.}
\begin{remark}
In the common notation, each representation is specified by its rule on
the mixed regions together with the representing components on the
remaining prescribed domains.


For $\widetilde{\mathcal{U}}_{\mathrm{top}}$, the values on the relevant
mixed region are given by
\[
U(x,a)=\iota(x)
\qquad
((x,a)\in\widehat B_e\times A_e),
\]
and the representing components  consist of an interior operator $\iota$, a
lower component $U_L$, and an upper operation $Q$. Dually, for
$\widetilde{\mathcal{U}}_{\mathrm{bot}}$,
\[
U(x,b)=\kappa(x)
\qquad
((x,b)\in\widehat A_e\times B_e),
\]
and the required data are a closure operator $\kappa$, an upper
component $U_R$, and a lower operation $P$.
For $\widetilde{\mathcal{U}}_{\mathrm{max}}^{r}$, the corresponding
values are prescribed by
\[
U(a,x)=a
\qquad
((a,x)\in A_e\times\widehat B_e),
\]
while the remaining components consist of a lower component $T$, symmetric
lower sets $D$ and $E$, and a map
\[
R:\Omega_{D,E}^{-}\to\widehat B_e.
\]
Dually, $\widetilde{\mathcal{U}}_{\mathrm{min}}^{r}$ satisfies
\[
U(b,x)=b
\qquad
((b,x)\in B_e\times\widehat A_e),
\]
and is represented by an upper component $S$, symmetric upper sets
$D$ and $E$, and a map
\[
F:\Omega_{D,E}^{+}\to\widehat A_e.
\]
Thus, for $\widetilde{\mathcal{U}}_{\mathrm{top}}$ and $\widetilde{\mathcal{U}}_{\mathrm{bot}}$,  the values on the mixed regions
are determined by a unary operator, whereas for $\widetilde{\mathcal{U}}_{\mathrm{max}}^{r}$ and  $\widetilde{\mathcal{U}}_{\mathrm{min}}^{r}$,
they are prescribed to equal one of the two arguments. These two forms
of representation impose different constraints, and no inclusion
relation between the corresponding classes follows in general.
All four representations are based on the common decomposition 
$
X=A_e\mathbin{\cup}\{e\}\mathbin{\cup}B_e\mathbin{\cup}I_e^3
$ 
and
impose explicit compatibility conditions between its constituent
regions.


\end{remark}

The defining conditions on the mixed domains also yield the following
relations between the two pairs of classes.
\begin{proposition}\label{prop:cross-part-relations}
For every fixed bounded trellis and neutral element satisfying the standing
assumptions,

\begin{equation}\label{eq:21}
 \widetilde{\mathcal{U}}_{\mathrm{top}}\cap\widetilde{\mathcal{U}}_{\mathrm{max}}^{r}
 =\varnothing,
 \qquad
 \widetilde{\mathcal{U}}_{\mathrm{bot}}\cap\widetilde{\mathcal{U}}_{\mathrm{min}}^{r}
 =\varnothing.
\end{equation}
However,  even on a proper bounded trellis,   the operations $U_{\mathrm{min}}$ and $U_{\mathrm{max}}$ in
Example~\ref{ex:proper-trellis-r-classes} satisfy, respectively,

\begin{equation}\label{eq:22}
 U_{\mathrm{min}}\in
 \widetilde{\mathcal{U}}_{\mathrm{top}}^{\star}\cap
 \widetilde{\mathcal{U}}_{\mathrm{min}}^{r},
 \qquad
 U_{\mathrm{max}}\in
 \widetilde{\mathcal{U}}_{\mathrm{bot}}^{\star}\cap
 \widetilde{\mathcal{U}}_{\mathrm{max}}^{r}.
\end{equation}
\end{proposition}

\begin{proof}
To prove the first equality in~\eqref{eq:21}, suppose that 
$
U\in
\widetilde{\mathcal{U}}_{\mathrm{top}}
\cap\widetilde{\mathcal{U}}_{\mathrm{max}}^{r}.
$
Membership in $\widetilde{\mathcal{U}}_{\mathrm{top}}$ and its
mixed-region formula give 
$
U(1,0)\in\widehat B_e.
$
On the other hand, membership in
$\widetilde{\mathcal{U}}_{\mathrm{max}}^{r}$, together with
commutativity, yields 
$
U(1,0)=U(0,1)=0\in A_e.
$
This contradicts $A_e\cap\widehat B_e=\varnothing$.  
For the second equality, membership in
$\widetilde{\mathcal{U}}_{\mathrm{bot}}$
gives $U(0,1)\in\widehat A_e$, whereas membership in
$\widetilde{\mathcal{U}}_{\mathrm{min}}^{r}$ gives $U(0,1)=U(1,0)=1\in B_e$.

For (\ref{eq:22}), 
use the sets determined in
Example~\ref{ex:proper-trellis-r-classes}. The $U_{\mathrm{min}}$ table
satisfies
$U_{\mathrm{min}}(c,a)=c$ and $U_{\mathrm{min}}(1,a)=1$ for every
$a\in A_e$, and $c\vee1=1$; hence it satisfies
\textup{(A1)--(A3)}. The
$U_{\mathrm{max}}$ table satisfies $U_{\mathrm{max}}(x,1)=x$ for
$x\in\widehat A_e$, so \textup{(B1)--(B3)} follow, with $n_x=x$ on
$[0,e[$. Their memberships in
$\widetilde{\mathcal{U}}_{\mathrm{min}}^{r}$ and
$\widetilde{\mathcal{U}}_{\mathrm{max}}^{r}$, respectively, were already
verified in that example.
\end{proof}

The preceding discussion concerns the relationships among the four classes
within the bounded-trellis setting. To place these results in the broader
framework of bounded psosets, we now clarify their precise relationship with
the classes introduced in~\cite{KL26}.
\begin{remark}\label{rem:relation-with-KL26}
The relation between the classes considered here and those introduced on
bounded psosets in~\cite{KL26} should be understood in two different
directions. First, every bounded trellis is a bounded psoset, so the
underlying structures considered in the present paper form a subclass of
those studied in~\cite{KL26}.

For a fixed bounded trellis $\mathbb P$ and neutral element $e$, the defining
mixed-region conditions of
$\mathcal{U}_{\mathrm{max}}^{\mathrm{ps}}$ and
$\mathcal{U}_{\mathrm{min}}^{\mathrm{ps}}$ in~\cite{KL26} can be written as
\[
 U(x,a)=x
 \qquad
 \bigl((x,a)\in\widehat B_e\times A_e\bigr)
\]
and
\[
 U(x,b)=x
 \qquad
 \bigl((x,b)\in\widehat A_e\times B_e\bigr),
\]
respectively. On the other hand, the corresponding mixed-region values in
$\widetilde{\mathcal{U}}_{\mathrm{top}}$ and
$\widetilde{\mathcal{U}}_{\mathrm{bot}}$ are
\[
 U(x,a)=\iota_U(x)
 \quad\text{and}\quad
 U(x,b)=\kappa_U(x),
\]
where the uniquely recovered maps are
\[
 \iota_U(x)=U(0,x)
 \qquad\text{and}\qquad
 \kappa_U(x)=U(x,1).
\]
Consequently,
\[
 U\in\mathcal{U}_{\mathrm{max}}^{\mathrm{ps}}
 \quad\Leftrightarrow\quad
 U\in\widetilde{\mathcal{U}}_{\mathrm{top}}
 \ \text{and}\
 \iota_U=\operatorname{id}_{\widehat B_e},
\]
and, dually,
\[
 U\in\mathcal{U}_{\mathrm{min}}^{\mathrm{ps}}
 \quad\Leftrightarrow\quad
 U\in\widetilde{\mathcal{U}}_{\mathrm{bot}}
 \ \text{and}\
 \kappa_U=\operatorname{id}_{\widehat A_e}.
\]
Thus, on bounded trellises, the psoset classes
$\mathcal{U}_{\mathrm{max}}^{\mathrm{ps}}$ and
$\mathcal{U}_{\mathrm{min}}^{\mathrm{ps}}$ are the identity-operator
subcases of the present  classes $\widetilde{\mathcal{U}}_{\mathrm{top}}$ and
$\widetilde{\mathcal{U}}_{\mathrm{bot}}$,  respectively.

For the  classes $\widetilde{\mathcal{U}}_{\mathrm{max}}^{r}$ and  $\widetilde{\mathcal{U}}_{\mathrm{min}}^{r}$, the relation is more direct. By commutativity and
the common region notation,
\[
 \mathcal{U}_{\mathrm{max}r}^{\mathrm{ps}}
 =
 \widetilde{\mathcal{U}}_{\mathrm{max}}^{r},
 \qquad
 \mathcal{U}_{\mathrm{min}r}^{\mathrm{ps}}
 =
 \widetilde{\mathcal{U}}_{\mathrm{min}}^{r}
\]
whenever the underlying bounded psoset is a bounded trellis. Hence the
representation theorems obtained here provide complete and uniquely
recoverable descriptions of these two psoset classes under the additional
trellis assumption. Accordingly, the present results are neither a mere restriction nor a
direct generalization of those in~\cite{KL26}: the underlying structure
is more specialized, the classes
$\widetilde{\mathcal{U}}_{\mathrm{top}}$ and
$\widetilde{\mathcal{U}}_{\mathrm{bot}}$ are broader, and the
representations of
$\widetilde{\mathcal{U}}_{\mathrm{max}}^{r}$ and
$\widetilde{\mathcal{U}}_{\mathrm{min}}^{r}$ additionally provide
unique recovery of their components.

\end{remark}

\begin{corollary}[Bounded-lattice specialization]
\label{cor:four-class-lattice-specialization}
Assume that $\unlhd$ is transitive, and put
\[
L=(X,\unlhd,\wedge,\vee,0,1).
\]
Then $L$ is a bounded lattice, every
$e\in X\setminus\{0,1\}$ belongs to $K$, and
\[
 I_e^1=I_e^2=\varnothing,\qquad I_e^3=I_e.
\]
Consequently,
\[
 \begin{gathered}
 A_e=[0,e[,\qquad B_e=]e,1],\\
 \widehat A_e=[0,e[\cup I_e=X\setminus[e,1],\qquad
 \widehat B_e=]e,1]\cup I_e=X\setminus[0,e],
 \end{gathered}
\]
and the six trellis classes considered in Part~I and Part~II coincide with
their corresponding untilded bounded-lattice classes. Under these
identifications, the construction and representation theorems of Part~I and
Theorems~\ref{thm:r-max-representation} and
\ref{thm:r-min-representation} specialize to valid representations of the
four principal bounded-lattice classes. Their component domains and recovery
formulas take the bounded-lattice forms studied in~\cite{ZQ26tb,ZQ26r}.
\end{corollary}

\begin{proof}
Under transitivity, a finite directed path from $x$ to $y$ contracts to the
direct comparison $x\unlhd y$, while every direct comparison is already a
path. Thus $x\lesssim y\Leftrightarrow x\unlhd y$ for all
$x,y\in X$. Elements incomparable
with $e$ neither reach $e$ nor are reachable from $e$, which proves the
displayed identities. A transitive pseudo-order is a partial order; since the trellis
already has binary meets and joins, $L$ is a bounded lattice. Transitivity and
antisymmetry also exclude non-trivial cycles, so every choice of
$e\in X\setminus\{0,1\}$ satisfies the standing condition.

Substituting the displayed sets into Definitions~\ref{def:top-bottom-classes}
and~\ref{def:umaxr-uminr-trellis} gives the asserted class identifications.
The product pseudo-order becomes the ordinary product order. By
Definition~\ref{def:regional-order-compatibility}, regional order compatibility
then supplies the comparisons required for monotonicity when the two
pairs belong to different regions, while
the associativity conditions and all recovery formulas are unchanged. Direct
substitution in the four representation formulas proves the result.
\end{proof}

\section{Conclusion}\label{sec:conclusion}

In this paper, under the assumption
$e\in K\setminus\{0,1\}$, we investigated the representation of the classes
$\widetilde{\mathcal{U}}_{\mathrm{top}}$,
$\widetilde{\mathcal{U}}_{\mathrm{bot}}$,
$\widetilde{\mathcal{U}}_{\mathrm{max}}^{r}$, and
$\widetilde{\mathcal{U}}_{\mathrm{min}}^{r}$ of uninorms on bounded
trellises. The main contributions can be summarized as follows.
\begin{enumerate}[label=\textup{(\roman*)},leftmargin=*,itemsep=2pt]
\item For $e\in K\setminus\{0,1\}$, we adopt and systematically use the
decomposition of $I_e$ into $I_e^1,I_e^2$, and $I_e^3$. It determines the
sets $A_e$, $B_e$, $\widehat A_e$, and $\widehat B_e$ used throughout both
parts. Proposition~\ref{prop:neutral-middle-transitive} also shows that the
neutral element $e$ of a uninorm is necessarily middle-transitive,
whereas the condition $e\in K\setminus\{0,1\}$ is an additional
restriction imposed in this study.


\item For $\widetilde{\mathcal{U}}_{\mathrm{top}}$, we establish explicit compatibility
conditions under which an interior operator, a lower component uninorm, and
an increasing, associative, and commutative upper operation combine to form a
uninorm. We also prove the converse and show that all three components are
uniquely determined by the resulting uninorm. A dual analysis gives the
corresponding representation of $\widetilde{\mathcal{U}}_{\mathrm{bot}}$.
Two additional conditions for each subclass characterize
$\widetilde{\mathcal{U}}_{\mathrm{top}}^{\star}$ and
$\widetilde{\mathcal{U}}_{\mathrm{bot}}^{\star}$.

\item \rev{For $\widetilde{\mathcal{U}}_{\mathrm{max}}^{r}$, we determine the
parts of the operation forced by the class definition and isolate, within
$I_e^3\times I_e^3$, the pairs whose values lie in $A_e$ or equal $e$.
Pairs whose values lie in $A_e$ have a common value, while pairs whose value
is $e$ satisfy a symmetric unique-partner condition; all other pairs in
$\widehat B_e\times\widehat B_e$ determine a partial operation with values
in $\widehat B_e$. These components give an exact piecewise representation
and are uniquely recoverable from the uninorm.  Once the component maps are increasing, regional order compatibility
precisely captures the comparisons between distinct regions required
for the piecewise-defined operation to be increasing.
Together with all the other admissibility
conditions, the regional order-compatibility and associativity requirements
are necessary and sufficient for the assembled operation to belong to the
stated class. The case $\Delta_U^-=E_U^-=\varnothing$ follows as a corollary,
and the corresponding statements hold for
$\widetilde{\mathcal{U}}_{\mathrm{min}}^{r}$.}

\item We provide four finite examples on proper bounded trellises, each
serving a distinct purpose. Example~\ref{ex:finite-utop-nonstar}
exhibits a uninorm in
$\widetilde{\mathcal{U}}_{\mathrm{top}}$ that does not belong to
$\widetilde{\mathcal{U}}_{\mathrm{top}}^{\star}$, thereby proving that
$
\widetilde{\mathcal{U}}_{\mathrm{top}}^{\star}
\subsetneq
\widetilde{\mathcal{U}}_{\mathrm{top}}
$
can occur. Example~\ref{ex:finite-utop-star} illustrates the construction
and unique recovery of the components for
$\widetilde{\mathcal{U}}_{\mathrm{top}}^{\star}$ and clarifies the role
of the additional join condition. Example~\ref{ex:regional-compatibility}
shows that the conditions imposed separately on the component operations
do not, by themselves, guarantee monotonicity between different cases of
the piecewise definition. Finally,
Example~\ref{ex:proper-trellis-r-classes} realizes nonempty sets
$\Delta_U^-$, $\Delta_U^+$, and $E_U^-$; its dual construction
also yields an example with $E_U^+\ne\varnothing$. We also compare the four classes, clarify the precise
relationship between the classes considered here and the corresponding
classes of uninorms on bounded psosets introduced in~\cite{KL26},
  and show that their
representations specialize to the bounded-lattice setting in the transitive
case.

\end{enumerate}

Several problems remain for future research. It would be useful to determine
the logical relationships among the compatibility and associativity
conditions in the representation theorems and to identify smaller equivalent
systems when dependencies occur. Another
problem is to characterize the symmetric sets $E\subseteq I_e^3\times I_e^3$
satisfying the unique-partner condition that can occur as $E_U^\pm$, and to
obtain simpler sufficient conditions for the partial operations $R_U^\pm$.
Effective enumeration methods for finite trellises may help with both
questions. The present results concern the four specified classes under the
assumption $e\in K\setminus\{0,1\}$ and do not characterize all uninorms on
an arbitrary bounded trellis. By
Proposition~\ref{prop:neutral-middle-transitive}, middle-transitivity is
automatic for every neutral element in the adopted jointly increasing sense;
non-extremal neutral elements lying in non-trivial cycles remain outside the
present framework.
Proposition~\ref{prop:r-nonempty} gives exact nonemptiness criteria for
$\widetilde{\mathcal{U}}_{\mathrm{max}}^{r}$ and
$\widetilde{\mathcal{U}}_{\mathrm{min}}^{r}$, whereas analogous criteria stated
solely in terms of the bounded trellis $\mathbb P$ and a prescribed
element $e\in K\setminus\{0,1\}$
for
$\widetilde{\mathcal{U}}_{\mathrm{top}}$,
$\widetilde{\mathcal{U}}_{\mathrm{bot}}$,
$\widetilde{\mathcal{U}}_{\mathrm{top}}^{\star}$, and
$\widetilde{\mathcal{U}}_{\mathrm{bot}}^{\star}$ remain to be established.

\end{document}